\documentclass[leqno]{amsart}
\usepackage{latexsym}
\usepackage{amssymb, amsbsy, amsthm, amsmath, amstext, amsopn, verbatim}
\usepackage[color,all]{xy}
\usepackage{amsfonts}
\usepackage{amscd}
\usepackage{mathrsfs}
\usepackage{verbatim}
\usepackage{xcolor}
\usepackage{tabularx}
\usepackage{mathtools}
\usepackage{hyperref}
\usepackage{enumerate}

\input xy
\xyoption{all}

\usepackage{parcolumns}

\newcommand{\bd}{\begin{displaymath}}
\newcommand{\ed}{\end{displaymath}}
\newcommand{\bcd}{\begin{CD}}
\newcommand{\ecd}{\end{CD}}

\newcommand{\on}{\operatorname}

\title[Tamely ramified geometric Langlands in positive characteristic]{tamely ramified geometric Langlands correspondence in positive characteristic}
\author{Shiyu Shen}
\address{The Institute of Science and Technology Austria\\Klosterneuburg 3400, Austria}
\email{shiyu.shen@ist.ac.at}
\thanks{
Shiyu Shen has received funding from NSF grant DMS-1502125 and the European Union's Horizon 2020 research and innovation program under the Marie Skłodowska-Curie grant agreement No. 101034413.
}

\numberwithin{equation}{section}
\newtheorem{theorem}{Theorem}[section]
\newtheorem{examples}[theorem]{Examples}
\newtheorem{lemma}[theorem]{Lemma}

\newtheorem{prop}[theorem]{Proposition}
\newtheorem{corollary}[theorem]{Corollary}

\theoremstyle{definition}
\newtheorem{definition}[theorem]{Definition}

\theoremstyle{remark}
\newtheorem{remark}[theorem]{Remark}

\begin{document}
\maketitle

\begin{abstract}
    We prove a version of the tamely ramified geometric Langlands correspondence in positive characteristic for $GL_n(k)$, where $k$ is an algebraically closed field of characteristic $p> n$. Let $X$ be a smooth projective curve over $k$ with marked points, and fix a parabolic subgroup of $GL_n(k)$ at each marked point. We denote by $\on{Bun}_{n,P}$ the moduli stack of (quasi-)parabolic vector bundles on $X$, and by $\mathcal{L}oc_{n,P}$ the moduli stack of parabolic flat connections such that the residue is nilpotent with respect to the parabolic reduction at each marked point. We construct an equivalence between the bounded derived category $D^{b}(\on{QCoh}({\mathcal{L}oc_{n,P}^{0}}))$ of quasi-coherent sheaves on an open substack $\mathcal{L}oc_{n,P}^{0}\subset\mathcal{L}oc_{n,P}$, and the bounded derived category $D^{b}(\mathcal{D}^{0}_{{\on{Bun}}_{n,P}}\on{-mod})$ of $\mathcal{D}^{0}_{{\on{Bun}}_{n,P}}$-modules, where $\mathcal{D}^0_{\on{Bun}_{n,P}}$ is a localization of $\mathcal{D}_{\on{Bun}_{n,P}}$ the sheaf of crystalline differential operators on $\on{Bun}_{n,P}$. Thus we extend the work of Bezrukavnikov-Braverman \cite{BB} to the tamely ramified case. We also prove a correspondence between flat connections on $X$ with regular singularities and meromorphic Higgs bundles on the Frobenius twist $X^{(1)}$ of $X$ with first-order poles.
\end{abstract}

\section{Introduction}\label{section: introduction} 
\subsection{Geometric Langlands in positive characteristic}
Let $X$ be a smooth projective curve over $\mathbb{C}$. Let $G$ be a reductive group over $\mathbb{C}$ and let $\breve{G}$ be its Langlands dual group. The geometric Langlands correspondence (GLC), as proposed by Beilinson and Drinfeld in \cite{BD}, is a conjectural equivalence between the (appropriately defined) category of $\mathcal{D}$-modules on the moduli stack $\on{Bun}_G$ of $G$-bundles on $X$, and the (appropriately defined) category of quasi-coherent sheaves on the moduli stack $\mathcal{L}oc_{\breve{G}}$ of $\breve{G}$-local systems on $X$. A precise statement of this conjecture can be found in \cite{AG15}.

In \cite{BB}, a generic version of the GLC in positive characteristic is established for $G=GL_n(k)$. The $\mathcal{D}$-modules are interpreted in terms of crystalline differential operators. Using the Azumaya property of crystalline differential operators and a twisted version of the Fourier-Mukai transform, the authors prove a generic version of the GLC over the open subset of the Hitchin base where the spectral curves are smooth. In the case of $G=GL_n(k)$, the results of \cite{BB} are generalized in various directions. In \cite{Nevins}, the mirabolic version of this correspondence is established. In \cite{Travkin}, the author proved the quantum version of this correspondence. In \cite{Grochenig}, the equivalence in \cite{BB} is extended to the Hitchin base of reduced and irreducible spectral curves. The results of \cite{BB} were extended to arbitrary reductive groups in \cite{CZ15} and \cite{CZ17}.

\subsection{Tamely ramified geometric Langlands correspondence}

The main purpose of this paper is to establish the tamely ramified version of the GLC proved in \cite{BB}, i.e. we allow the flat connections to have regular singularities. The term ``tamely ramified" comes from analogy with the local Langlands program. See \cite{Frenkel} Section 8 for a discussion of the tamely ramified GLC over $\mathbb{C}$. Let $k$ be an algebraically closed field of characteristic $p$, and let $X$ be a smooth projective curve over $k$. We will work on the case of $G=GL_n(k)$ and assume $p> n$. Let $D=q_1+q_2+\cdots+q_m$ be an effective reduced divisor on $X$, and let $P_D=(P_1, P_2,\dots,P_m)$ be an ordered $m$-tuple of parabolic subgroups of $GL_n(k)$. We assume that we are in one of the following three cases:
\begin{enumerate}
    \item $g_X\geq 2$,
    \item $g_X=1$, $m\geq 2$ and at least two $P_i$ are proper parabolic subgroups; or $m=1$, $P$ is a Borel subgroup and $n\geq 3$,
    \item $g_X=0$, $m\geq 4$ and all $P_i$ are Borel subgroups.
\end{enumerate}

Compared to the unramified version of the GLC, instead of considering $\on{Bun}_n$ and $\mathcal{L}oc_n$, we consider the moduli stack $\on{Bun}_{n,P_D}$ of (quasi-)parabolic vector bundles (vector bundles of rank $n$ with a $P_i$-reduction at each $q_i$), and the moduli stack $\mathcal{L}oc_{n,P_D}$ of flat connections on parabolic vector bundles with regular singularities at $q_1, q_2, \dots, q_m$ such that the residue at each $q_i$ is nilpotent with respect to the $P_i$-reduction. Note that the cotangent bundle $T^{*}\on{Bun}_{n,P_D}$ is isomorphic to the moduli stack $\mathcal{H}iggs_{n,P_D}$ of parabolic Higgs bundles such that the residue of the Higgs field is nilpotent with respect to the parabolic reduction at each $q_i$. We denote by $\mathcal{D}_{\on{Bun}_{n,P_D}}$ the sheaf of crystalline differential operators on $\on{Bun}_{n,P_D}$\footnote{Defined in the sense of \cite{BB}, Section 3.13. See Section \ref{Section: differential operators on stacks} and Section \ref{subsection: statement of the main theorem}.}. We will define a localization $\mathcal{D}^0_{\on{Bun}_{n,P_D}}$ of $\mathcal{D}_{\on{Bun}_{n,P_D}}$ and an open substack $\mathcal{L}oc^0_{n,P_D}$ of $\mathcal{L}oc_{n,P_D}$ (see Section \ref{subsection: statement of the main theorem} for precise definitions). We will construct an $\mathcal{O}_{\mathcal{L}oc_{n,P_D}^{0}}\boxtimes \mathcal{D}^{0}_{{\on{Bun}}_{n,P_D}}$-module $\mathcal{P}$ (see Section \ref{subsection: proof of main theorem}) and consider the Fourier-Mukai functor with kernel $\mathcal{P}$
\[
    \Phi_{\mathcal{P}}: D^{b}(\on{QCoh}({\mathcal{L}oc_{n,P_D}^{0}}))\longrightarrow D^{b}(\mathcal{D}^{0}_{{\on{Bun}}_{n,P_D}}\on{-mod})
\]
from the bounded derived category of quasi-coherent sheaves on $\mathcal{L}oc_{n,P_D}^{0}$ to the bounded derived category of $\mathcal{D}^{0}_{{\on{Bun}}_{n,P_D}}$-modules.
The main theorem of the paper is the following:
\begin{theorem}\label{Thm: main theorem intro}
    $\Phi_{\mathcal{P}}$ is an equivalence of derived categories. 
\end{theorem}

There are natural functors from both sides of the equivalence: the Hecke functor $\on{H}_{P_D}^0$ (see Section \ref{subsection: Hecke})
\[
    \on{H}_{P_D}^0: D^{b}(\mathcal{D}^0_{\on{Bun}_{n,P_D}}\on{-mod})\longrightarrow D^{b}(\mathcal{D}^0_{\on{Bun}_{n,P_D}}\boxtimes \mathcal{D}_{X\backslash D}\on{-mod})
\]
and the functor $\on{W}^0_{P_D}$
\[
    \on{W}^0_{P_D}: D^{b}(\mathcal{O}_{\mathcal{L}oc^0_{n,P_D}}\on{-mod})\longrightarrow D^{b}(\mathcal{O}_{\mathcal{L}oc^0_{n,P_D}}\boxtimes \mathcal{D}_{X\backslash D}\on{-mod})
\]
defined by tensoring with the universal flat connection. Let $\Phi_{\mathcal{P},X\backslash D}$ be the Fourier-Mukai equivalence induced by the pull-back of $\mathcal{P}$:
\[
\Phi_{\mathcal{P},X\backslash D}: D^{b}(\mathcal{O}_{\mathcal{L}oc^0_{n,P_D}}\boxtimes \mathcal{D}_{X\backslash D}\on{-mod})\xrightarrow{\simeq} D^{b}(\mathcal{D}^0_{\on{Bun}_{n,P_D}}\boxtimes \mathcal{D}_{X\backslash D}\on{-mod}).
\]
The equivalence in Theorem \ref{Thm: main theorem intro} satisfies the following Hecke eigenvalue property:
\begin{theorem}\label{thm: hecke intro}
There is an isomorphism of functors:
        \[\on{H}_{P_D}^0\circ \Phi_{\mathcal{P}}\cong\Phi_{\mathcal{P},X\backslash D}\circ\on{W}^0_{P_D} .\]
\end{theorem}

Now let $(E,\nabla)$ be a $k$-point of $\mathcal{L}oc^0_{n,P_D}$. We denote by $\mathcal{M}_{E,\nabla}$ the image of $(E,\nabla)$ under $\Phi_{\mathcal{P}}$. By Theorem \ref{thm: hecke intro}, $\mathcal{M}_{E,\nabla}$ satisfies
\begin{equation*}\label{eq: hecke}
       \on{H}^0_{P_D}(\mathcal{M}_{E,\nabla})\cong\mathcal{M}_{E,\nabla}\boxtimes E. 
\end{equation*}

\subsection{Summary of the proof}
We fix a $k$-point $q\in X$ and a parabolic subgroup $P$ of $GL_n(k)$. For the purpose of simplifying notations, our exposition will be restricted to the case of $D=q$ and $P_D=P$ from now on. The only proof that will be different in the more general setting is the proof of Proposition \ref{prop: the open subset is non=empty} in the case of $X=\textbf{P}^1$, $m\geq 4$ and all $P_i$ are Borel subgroups. We discuss this case in Remark \ref{Remark: P1}.

Our proof of Theorem \ref{Thm: main theorem intro} is based on the same strategy as used in \cite{BB}, but some new ingredients come into play. Note that in \cite{BB}, the geometric Langlands correspondence is established over the open subset of the Hitchin base where the spectral curves are smooth. Compared to the unramified case in \cite{BB}, one of the main difficulties in the tamely ramified case is that unless $P$ is a Borel subgroup of $GL_n(k)$, there are no smooth spectral curves. We resolve this situation by considering the normalization of the spectral curves. It is observed in \cite{SS} that under generic restrictions on the spectral curves, a fiber of the Hitchin map 
\[
    h_P: \mathcal{H}iggs_{n,P}\longrightarrow B_P
\]
is isomorphic to the Picard stack of the normalization of the corresponding spectral curve. In Section \ref{subsection: spectral data}, we extend this observation to a family version. More precisely, we prove:
\begin{theorem}\label{thm: spectra data intro}
    There exists a Zariski open dense subset $B_P^0\subset B_P$ and a flat family of smooth projective curves $\widetilde{\Sigma}\longrightarrow B_P^0$ such that 
    \[
        \mathcal{H}iggs_{n,P}\times_{B_P}B^0_P\cong \on{Pic}(\widetilde{\Sigma}/B_P^0).
    \]
    For each $b\in B_P^0(k)$, $\widetilde{\Sigma}_b$ is the normalization of the spectral curve $\Sigma_b$.
\end{theorem}

In Section \ref{section:nah}, we establish a correspondence between flat connections on $X$ with regular singularity at $q$ and $\Omega_{X^{(1)}}(q)$-twisted Higgs bundles on the Frobenius twist $X^{(1)}$ of $X$, which can be thought of as a characteristic $p$ version of the non-abelian Hodge correspondence in \cite{Simpson}. Let $\underline{a}$ be an unordered $n$-tuple of elements in $k$. We denote by $\mathcal{H}iggs_{n,\underline{a}}(X^{(1)})$ the moduli stack of $\Omega_{X^{(1)}}(q)$-twisted Higgs bundles $(E,\phi)$ on $X^{(1)}$ such that the tuple of eigenvalues of the residue $\on{res}_{q}(\phi)$ of the Higgs field at $q$ is $\underline{a}$. Let $B_{\underline{a}}^{(1)}$ be the image of $\mathcal{H}iggs_{n,\underline{a}}(X^{(1)})$ under the Hitchin map $h^{(1)}$. We fix a set-theoretic section $\sigma$ of the Artin-Schreier map $k\longrightarrow k$ that maps $t$ to $t^p-t$. We denote by $\mathcal{L}oc_{n,\sigma(\underline{a})}$ the moduli stack of flat connections $(E,\nabla)$ with regular singularity at $q$ such that the tuple of eigenvalues of $\on{res}_{q}({\nabla})$ is $\sigma(\underline{a})$. The $p$-curvature of $(E,\nabla)$ (see Section \ref{section: hitchin morphism for flat connections}) defines the Hitchin map $h'$ for flat connections with regular singularity at $q$:
\[
    h': \mathcal{L}oc_{n,\sigma(\underline{a})}\longrightarrow B_{\underline{a}}^{(1)}.
\]
We will define an open substack (see Section \ref{subsection: Statement of the theorem}) \[\mathcal{L}oc^r_{n,\sigma(\underline{a})}\subset \mathcal{L}oc_{n,\sigma(\underline{a})}
\]
and prove the following theorem:
\\
\begin{theorem}\label{thm: nah intro}
\mbox{}
   \begin{enumerate}
    \item $\mathcal{L}oc^r_{n,\sigma(\underline{a})}$ is a $\on{Pic}(\Sigma^{(1)}/B^{(1)}_{\underline{a}})$-torsor,
  
    \item $\mathcal{L}oc_{n, \sigma(\underline{a})}\cong \mathcal{L}oc_{n, \sigma(\underline{a})}^r\times^{\on{Pic}(\Sigma^{(1)}/B^{(1)}_{\underline{a}})}\mathcal{H}iggs_{n,\underline{a}}(X^{(1)})$. 
    \end{enumerate}
\end{theorem}

Note that for an arbitrary reductive group $G$, a similar construction is used in \cite{CZ15} to establish the characteristic $p$ version of the non-abelian Hodge correspondence for flat connections without singularities. 

One of the key steps in our proof of Theorem \ref{thm: nah intro} is to show that the map 
\[h':  \mathcal{L}oc^r_{n,\sigma(\underline{a})}\longrightarrow B_{\underline{a}}^{(1)}
\]
is surjective. Since we consider flat connections with singularity at $q$, we cannot apply the Azumaya property of differential operators on $X$ directly. Instead, we construct a flat connection on $X\backslash q$ using the Azumaya property, construct a flat connection on the formal disk around $q$ by explicitly solving a differential equation for the connection form, and glue them together using the Beauville-Laszlo theorem \cite{BL95}.

Note that $\mathcal{L}oc_{n,(\underline{0})}$ is the moduli stack of flat connections with regular singularity and nilpotent residue at $q$. Restricting the isomorphism in Theorem $\ref{thm: nah intro}$(2) to $(B_P^0)^{(1)}$ and combining Theorem $\ref{thm: spectra data intro}$, we deduce that $\mathcal{L}oc_{n,P}^{0}\coloneqq\mathcal{L}oc_{n,P}\times_{B_P^{(1)}}(B_P^{0})^{(1)}$ is a $\on{Pic}(\widetilde{\Sigma}^{(1)}/(B_P^{0})^{(1)})$-torsor. 

It is proved in \cite{BB} that for a smooth algebraic stack $Z$ that is good in the sense of \cite{BD}(i.e. $Z$ satisfies $\on{dim}T^{*}Z=2\on{dim}Z$), there is a natural sheaf of algebras $\mathcal{D}_Z$ on $T^{*}Z^{(1)}$ that satisfies $\pi^{(1)}_{*}\mathcal{D}_Z\cong \on{Fr}_{*}D_Z$, and the restriction of $\mathcal{D}_Z$ to the maximal smooth open substack $(T^{*}Z^{0})^{(1)}\subseteq (T^{*}Z)^{(1)}$ is an Azumaya algebra of rank $p^{2\on{dim} Z}$ (see Section \ref{Section: differential operators on stacks} for a review of this construcion). Here $\pi^{(1)}: T^{*}Z^{(1)}\longrightarrow Z^{(1)}$ is the projection and $\on{Fr}: Z\longrightarrow Z^{(1)}$ is the relative Frobenius. The stack $\on{Bun}_{n,P}$ ``almost'' satisfies those two properties, and we can still construct a sheaf of algebras $\mathcal{D}_{\on{Bun}_{n,P}}$ that satisfies $\pi^{(1)}_{*}\mathcal{D}_{\on{Bun}_{n,P}}\cong \on{Fr}_{*}D_{\on{Bun}_{n,P}}$. See Section \ref{subsection: definition of the algebra} for details.

The restriction $\mathcal{D}^{0}_{\on{Bun}_{n,P}}$ of $\mathcal{D}_{\on{Bun}_{n,P}}$ to  \[\mathcal{H}iggs_{n,P}^{(1)}\times_{B_P^{(1)}}(B_P^{0})^{(1)}\cong \on{Pic}(\widetilde{\Sigma}^{(1)}/(B_P^{0})^{(1)})
\]
is an Azumaya algebra. We associate with $\mathcal{D}^{0}_{\on{Bun}_{n,P}}$ its stack of splittings $\mathcal{Y}_{\mathcal{D}^{0}_{\on{Bun}_{n,P}}}$, which is a $\mathbb{G}_m$-gerbe over the Picard stack $\on{Pic}(\widetilde{\Sigma}^{(1)}/(B_P^{0})^{(1)})$. In Section \ref{subsection: group structure on Azumaya algebra}, we show that $\mathcal{D}^{0}_{\on{Bun}_{n,P}}$ has a tensor structure, therefore $\mathcal{Y}_{\mathcal{D}^{0}_{\on{Bun}_{n,P}}}$ has the structure of a commutative group stack, and there is a short exact sequence
\[
    0\longrightarrow B\mathbb{G}_m\longrightarrow \mathcal{Y}_{\mathcal{D}^0_{\on{Bun}_{n,P}}} \longrightarrow \on {Pic}(\widetilde{\Sigma}^{(1)}/(B_P^{0})^{(1)})\longrightarrow 0. 
\]
By taking dual, we get another short exact sequence:
\[
    0\longrightarrow \on {Pic}(\widetilde{\Sigma}^{(1)}/(B_P^{0})^{(1)})\longrightarrow \mathcal{Y}_{\mathcal{D}^0_{\on{Bun}_{n,P}}}^{\vee}\xrightarrow{\pi} \mathbb{Z}\longrightarrow 0. 
\]
In Section \ref{subsection: proof of main theorem}, we prove that $(\mathcal{Y}_{\mathcal{D}^0_{\on{Bun}_{n,P}}}^{\vee})_1\coloneqq \pi^{-1}(1)$ is isomorphic to $\mathcal{L}oc_{n,P}^{0}$ as  $\on{Pic}(\widetilde{\Sigma}^{(1)}/(B_P^{0})^{(1)})$-torsors, therefore we can apply a twisted version of the Fourier-Mukai transform (reviewed in Section \ref{subsection: FM}) to prove the equivalence in Theorem \ref{Thm: main theorem intro}. For the proof of this isomorphism, we show that the tautological 1-form $\theta^{(1)}$ on $T^{*}(X\backslash q)^{(1)}$ extends to a 1-form $\widetilde{\theta}^{(1)}$ on $\widetilde{\Sigma}^{(1)}$, and both $\on{Pic}(\widetilde{\Sigma}^{(1)}/(B_P^{0})^{(1)})$-torsors are isomorphic to the moduli stack of rank one flat connections on $\widetilde{\Sigma}$ with $p$-curvature $\widetilde{\theta}^{(1)}$. 

\subsection{Structure of the article}
In Section \ref{section: Spectral data of parabolic Higgs bundles}, we first review some basic constructions related to the Hitchin fibration. Then we define the Zariski open dense subset $B_P^0\subset B_P$ and establish the correspondence between parabolic Higgs bundles and the Picard stack of the normalization of spectral curves over $B_P^0$. In Section \ref{section: Azumaya property of differential operators }, we first review some properties of crystalline differential operators in positive characteristic, including the Azumaya property and the Cartier descent. Then we describe the correspondence between modules over an Azumaya algebra and twisted sheaves associated to its $\mathbb{G}_m$-gerbe of splittings. Finally we review the definition of tensor structures on Azumaya algebras over group stacks. In Section \ref{section:nah}, we first construct the Hitchin map for flat connections with regular singularities. Then we prove the non-abelian Hodge correspondence between $\mathcal{L}oc_{n,\sigma(\underline{a})}$ and $\mathcal{H}iggs_{n,\underline{a}}(X^{(1)})$.  
In Section \ref{section: main theorem}, we first define the sheaf of algebras $\mathcal{D}_{\on{Bun}_{n,P}}$ and construct a tensor structure on $\mathcal{D}^{0}_{\on{Bun}_{n,P}}$. Then we review the Fourier-Mukai transforms on commutative group stacks and use this framework to prove the main theorem. Finally we discuss the Hecke eigenvalue property of this equivalence.

\subsection{Notations and definitions}
Unless otherwise mentioned, $k$ is an algebraically closed field of characteristic $p>0$. We consider the general linear group $GL_n(k)$ and assume $p> n$. Let $\mathfrak{gl}_n(k)$ be the Lie algebra of $GL_n(k)$. We denote by $\mathcal{N}$ the nilpotent cone in $\mathfrak{gl}_n(k)$. Let $P$ be a parabolic subgroup of $GL_n(k)$. The Lie algebra of $P$ decomposes as $Lie(P)\cong l\oplus n_{P}^+$. We denote by $\mathcal{O}_P$ the Richardson orbit corresponding to $P$, which is the unique nilpotent orbit in $\mathfrak{gl}_n(k)$ such that the intersection with $n_{P}^+$ is open dense in $n_{P}^+$. Let $X$ be a smooth projective algebraic curve over $k$. Let $g_X$ be the genus of $X$. We fix a $k$-point $q\in X$. For any $k$-scheme $S$, we denote by $\iota_q: S\longrightarrow S\times X$ the base change of $q: \on{Spec}(k)\longrightarrow X$. We denote by $p_X$ the projection from $S\times X$ to $X$, and by $p_S$ the projection to $S$. 

\begin{definition}
	An $S$-family of (quasi-)parabolic vector bundles on $X$ is a vector bundle $E$ of rank $n$ on $S\times X$ with a $P$-reduction along $S\times q$. We denote the moduli stack of such objects by $\on{Bun}_{n,P}$. To be more precise, $\on{Bun}_{n,P}$ classifies triples $(E, E_P,\tau)$, where $E_P$ is a $P$-bundle on $S\times q$ and $\tau$ is an isomorphism
	\[
	    \tau: E_P\times_{P}k^n\xrightarrow{\simeq} \iota^{*}_{q} E.
	\]
	
	Let $\on{Bun}_n$ be the moduli stack of rank $n$ vector bundles on $X$. There is a canonical map from $\on{Bun}_{n,P}$ to $\on{Bun}_n$, which is defined by forgetting the $P$-reduction. 
\end{definition}

\begin{remark}\label{remark: parabolic subgroup}
    Let $B$ be the Borel subgroup of $GL_n(k)$ that consists of upper triangular matrices. There is a one-to-one correspondence between the set of parabolic subgroups of $GL_n(k)$ containing $B$ and the set of ordered $n$-tuples of positive integers $\mu=(\mu_1,\mu_2,\dots,\mu_s)$ such that $\displaystyle\sum_{i=1}^s\mu_i=n$. This correspondence can be described as follows. We consider the standard representation of $GL_n(k)$ acting on $k^n$. Let $e_1, e_2,\dots,e_n$ be the standard basis of $k^n$. For $i=1,2,\dots,s$, let $V_i=\displaystyle\bigoplus_{j=1}^{m_i}ke_j$ where $m_i=\displaystyle\sum_{k=1}^i \mu_k$. Then the parabolic subgroup $P_{\mu}$ corresponding to $\mu$ is identified with
    \[
      \{g\in GL_n(k)| g(V_i)\subseteq V_i, 1\leq i\leq s       \}.
    \]
    Let $\lambda_1\geq\lambda_2\geq\cdots\geq\lambda_r$ be the conjugate partition to $\mu$. The Richardson orbit corresponding to $P_\mu$ consists of $M_{\underline{\lambda}}$ the nilpotent matrix with Jordan blocks of sizes $\lambda_1\geq \lambda_2\geq \cdots\geq \lambda_r$.
    
    Let $E$ be a rank $n$ vector bundle on $S\times X$. A $P_\mu$-reduction of the structure group along $S\times q$ corresponds to a partial flag structure:
    \[
        0=E^0_q\subset E^1_q\subset E^2_q \subset \cdots E^s_q=\iota_q^{*}E,
    \]
    where $E_i$ is a vector bundle of rank $m_i$ on $S$. 
\end{remark}

\begin{remark}
  In the work of Mehta-Seshadri \cite{MS}, a parabolic vector bundle is defined as a quasi-parabolic vector bundle together with a set of real numbers $(\alpha_0, \alpha_1, \alpha_2, \dots, \alpha_s)$ satisfying \[ 1=\alpha_0>\alpha_1>\cdots>\alpha_r\geq 0\]
  called parabolic weights. The parabolic weights can be used to define a stability condition on such objects, which is necessary for the construction of a moduli space. Since we focus on studying the moduli stack of such objects, we do not introduce the parabolic weights in this paper. 
\end{remark}

\begin{definition}\label{def: parabolic Higgs}
	An $S$-family of (quasi-)parabolic Higgs bundles on $X$ is a parabolic vector bundle $(E, E_P,\tau)$ together with a Higgs field 
	\[\phi\in \Gamma(\mathcal{E}nd(E)\otimes p_X^{*}(\Omega_{X}(q))), 
	\]such that the residue of $\phi$ at $q$, which we denote by $\on{res}_{q}(\phi)\in \on{End}(\iota^{*}_q E)$, lies in $\Gamma(S,E_P\times_P n_{P}^+)$. In other words, if the parabolic reduction gives the following partial flag structure:
	 \[
        0=E_0\subset E_1\subset E_2 \subset \cdots E_s=\iota_q^{*}E,
    \]
    we require $\on{res}_q(\phi)(E_i)\subseteq E_{i-1}$.
	We denote the moduli stack of such objects by $\mathcal{H}iggs_{n,P}$.
	
	We denote by $\mathcal{H}iggs_{n,q}$ the moduli stack of $\Omega_{X}(q)$-twisted Higgs bundles $(E,\phi)$, 
	\[\phi\in \Gamma(\mathcal{E}nd(E)\otimes p_X^{*}(\Omega_{X}(q))).
	\]
	There is a canonical map from $\mathcal{H}iggs_{n,P}$ to $\mathcal{H}iggs_{n,q}$, which is defined by forgetting the $P$-reduction.  
\end{definition}

\begin{remark}

$\mathcal{H}iggs_{n,P}\cong T^{*}{\on{Bun}}_{n,P}$. 
	
\end{remark}

\begin{remark}
    A parabolic version of the Hitchin moduli stack is previously considered in the work of Yun \cite{Yun}. Definition 2.1.1 in \cite{Yun} is different from our Definition \ref{def: parabolic Higgs} in two aspects: the marked point $q$ on $X$ is allowed to move in \cite{Yun}, and the Higgs field $\phi$ is only required to preserve the flag structure instead of being nilpotent with respect to the flag structure. 
\end{remark}

\begin{definition}
	An $S$-family of parabolic flat connections on $X$ is a parabolic vector bundle $(E, E_P,\tau)$ together with a flat connection with regular singularity at $q$
	\[\nabla:E\longrightarrow E\otimes p_X^{*}(\Omega_{X}(q)),
	\]
	(i.e. $\nabla$ is a $\mathcal{O}_S$-linear map of sheaves that satisfies the Leibniz rule), such that the residue $\on{res}_q{\nabla}$ of $\nabla$ at $q$ lies in $\Gamma(S,E_P\times_P n_{P}^+)$. We denote the moduli stack of such objects by $\mathcal{L}oc_{n,P}$.
	
	We denote by $\mathcal{L}oc_{n,q}$ the moduli stack of flat connections of rank $n$ on $X$ with regular singularity at $q$. There is a canonical map from $\mathcal{L}oc_{n,P}$ to $\mathcal{L}oc_{n,q}$, which is defined by forgetting the $P$-reduction. 
\end{definition}

\subsection{Acknowledgements}
I would like to thank my advisor Tom Nevins for many helpful discussions on this subject, and for his comments on this paper. I would like to thank Christopher Dodd, Michael Groechenig and Tamas Hausel for helpful conversations. I would like to thank Tsao-Hsien Chen and Siqing Zhang for useful comments on an earlier version of this paper. 

\section{Spectral data of parabolic Higgs bundles}\label{section: Spectral data of parabolic Higgs bundles}
\subsection{Basic constructions}\label{subsection: basic constructions}

In this Subsection, we discuss the construction of the Hitchin map, spectral curves and spectral sheaves in \cite{Hitchin} and \cite{BNR} in the parabolic setting. By taking the coefficients of the characteristic polynomial of the Higgs field, we get the Hitchin map:
\[h:\mathcal{H}iggs_{n,q}\longrightarrow B,\]
where $B=\displaystyle\bigoplus_{i=1}^n\Gamma(X,\Omega_{X}(q)^i)$\footnote{More precisely, $B$ is the affine space associated to the $k$-vector space $\bigoplus_{i=1}^n\Gamma(X,\Omega_{X}(q)^i)$, i.e. $B=\on{Spec}(\on{Sym}(\bigoplus_{i=1}^n\Gamma(X,\Omega_{X}(q)^i))^{\vee})$.}. If we require the residue of the Higgs field to be nilpotent, the image of this map lies in $B_{\mathcal{N}}\coloneqq\displaystyle\bigoplus_{i=1}^n\Gamma(X,\Omega^{\otimes i}_{X}((i-1)q)$.

Let $T^{*}X(q)=\on{Spec}_X(\on{Sym}_{\mathcal{O}_X}\mathcal{T}_X(-q))$, where $\mathcal{T}_X(-q)$ is the sheaf of vector fields on $X$ that vanish at $q$. Let $\pi$ be the projection $\pi: T^{*}X(q)\longrightarrow X$. We denote by $y$ the tautological section of $\pi^{*}(\Omega_X(q))$. For $b=(b_1,b_2,\dots,b_n)$, $b_i\in\Gamma(X,\Omega_{X}(q)^{i})$, we define the spectral curve $\Sigma_b$ to be the zero-subscheme of the section 
\[
    y^n+b_1y^{n-1}+\cdots+b_{n-1}y+b_n
\]
of $\pi^{*}(\Omega_X(q)^n)$. By abuse of notation, we also denote by $\pi$ the projection from $\Sigma_b$ to $X$. Since $\pi_{*}\mathcal{O}_{\Sigma_b}=\displaystyle \bigoplus_{i=0}^{{n-1}}\mathcal{T}_X(-q)^{\otimes i}$, we can compute the genus of $\Sigma_b$
\[
    g_{\Sigma_b}=\frac{n(n-1)}{2}(2g-1)+n(g-1)+1. 
\]

Let $(E,\phi)$ be a $k$-point of $\mathcal{H}iggs_{n,q}$ such that $h(E,\phi)=b$. We can think of $\phi$ as a morphism
\[
    \phi: \mathcal{T}_X(-q)\longrightarrow \mathcal{E}nd(E).
\]
By Cayley-Hamilton, there is a coherent sheaf $\mathcal{F}$ on $\Sigma_b$ such that $\pi_{*}(\mathcal{F})=E$. We call $\mathcal{F}$ the spectral sheaf corresponding to $(E,\phi)$. Conversely, let $\mathcal{G}$ be a coherent sheaf on $\Sigma_b$, there is a canonical section $\phi_{\on{can}}\in \Gamma(X, \mathcal{E}nd(\pi_{*}(\mathcal{G}))\otimes \Omega_X(q))$ obtained by adjunction. It is proved in \cite{BNR} that if $\Sigma_b$ is reduced, the Hitchin fiber $h^{-1}(b)$ is isomorphic to the stack of torsion free sheaves on $\Sigma_b$, and if $\Sigma_b$ is smooth, $h^{-1}(b)$ is isomorphic to the Picard stack $\on{Pic}(\Sigma_b)$ of $\Sigma_b$.

If we require the residue of $\phi$ to be nilpotent, then $\pi^{-1}(q)$ is a single point $q'$ that lies in the zero-section of $T^{*}X(q)$. Let $V=\on{spec}(A)$ be an affine open neighborhood of $q$ in $X$. Let $x$ be an element of $A$ that is mapped to a local parameter of $X$ at $q$. Shrinking $V$ if necessary, we assume $\frac{dx}{x}$ is a nowhere vanishing section of $\Omega_V(q)$. Let $U=\pi^{-1}(V)$. The section $\frac{dx}{x}$ gives a trivialization of $T^{*}X(q)|_V$ and $\pi^{*}(T^{*}X(q))|_U$. Under this trivialization, the tautological section $y$ is equal to $x\partial_x$ considered as an element in $\mathcal{O}_{U}$. Let $\Sigma_b(V)\coloneqq V \times _X \Sigma_b$, then $\mathcal{O}_{\Sigma_b(V)}$ is isomorphic to $\mathcal{O}_U/(f_b)$, where $f_b=y^n+b_{1}y^{n-1}+\dots + b_{n-1}y+b_n$, $b_i\in \mathcal{O}_V$. We denote by $\hat{f}_b$ the image of $f_b$ in $\hat{\mathcal{O}}_{U,q'}\cong k[[x,y]]$, then $\hat{\mathcal{O}}_{\Sigma_{b},q'}\cong k[[x,y]]/(\hat{f}_b)$.

\subsection{The parabolic Hitchin base $B_P$}\label{subsection: Hitchin base}
Now let $P$ be a parabolic subgroup of $GL_n(k)$, and we assume the Richardson orbit $\mathcal{O}_P$ of $P$ contains the nilpotent matrix with Jordan blocks of sizes $\lambda_1\geq\lambda_2\geq\cdots\geq\lambda_r, \displaystyle\sum_{i=1}^r\lambda_i=n$. Composing $h$ with the forgetful map from $\mathcal{H}iggs_{n,P}$ to $\mathcal{H}iggs_{n,q}$, we get 
\[h_P:\mathcal{H}iggs_{n,P}\longrightarrow B.
\] 
In order to describe the image of $h_P$, we define the following sets of formal power series. 
\begin{definition}
Let $\boldsymbol\eta=(\eta_1, \eta_2,\dots,\eta_s)$, $\eta_1\geq\eta_2\geq\cdots\geq\eta_s$ be a decreasing sequence of positive integers. Let $\gamma_i=\displaystyle\sum_{j=i+1}^s\eta_j$ for $i=0,1,2,\dots,s-1$ and $\gamma_s=0$. We denote by $P_{\boldsymbol\eta}$ the set of formal power series of the form
\[ y^{\gamma_0}+\sum_{i=1}^{s}a_i(x,y)x^iy^{\gamma_i}, \text{where } a_i(x,y)\in k[[x,y]],
\]
and by $P_{\boldsymbol\eta}^{0}$ the subset of elements in $P_{\boldsymbol\eta}$ that satisfy $a_i(x,y)\in k[[x,y]]^{\times}$. 
\end{definition}
In particular, if $\boldsymbol\eta=(m)$, $P_m^0$ is the set of formal power series of the form
\[
    y^m+a(x,y)x, \text{where } a(x,y)\in k[[x,y]]^{\times}.
\]

\begin{lemma}\label{lemma: inclusion}
    Let $(E, E_P,\tau, \phi)$ be a $k$-point of $\mathcal{H}iggs_{n,P}$ such that $h_P(E,\phi)=b$. Let $\hat{f}_b$ be the element in $k[[x,y]]$ such that $\hat{\mathcal{O}}_{{\Sigma}_b,q'}\cong k[[x,y]]/(\hat{f}_b)$ as above, then $\hat{f}_b\in P_{\boldsymbol\lambda}$.
\end{lemma}

This lemma follows from a direct computation, see Proposition 22 in \cite{BK18}. It follows from Lemma \ref{lemma: inclusion} that $h_P$ factors through the affine space
\[B_P\coloneqq\displaystyle\bigoplus_{i=1}^{n}\Gamma(X, \Omega_X^{\otimes i}((i-m_i)q)),
\]
here $m_i=j$ if $\gamma_{j}\leq n-i < \gamma_{j-1}$. To be more explicit, we have $\boldsymbol{m}=1^{\lambda_1}2^{\lambda_2}\cdots r^{\lambda_r}$, meaning that the first $\lambda_1$ terms are $1$, the next $\lambda_2$ terms are $2$,..., and the last $\lambda_r$ terms are $r$.

\begin{lemma}\label{lemma: factor}
\mbox{}
\begin{enumerate}
    \item 
    Let $\hat{f}$ be a formal power series in $P_{\boldsymbol\eta}^{0}$, $\boldsymbol\eta=\eta_1^{l_1}\eta_2^{l_2}\cdots \eta_t^{l_t}$, $\eta_1\geq\eta_2\geq\cdots\geq\eta_t$. Then $\hat{f}$ factorizes uniquely as $\hat{f}=f_1f_2\cdots f_t$, where each $f_i$ is a formal power series in $P_{\eta_i^{l_i}}^{0}$,
    \item Let $\hat{g}=y^{\eta l}+a_1(x,y)xy^{\eta(l-1)}+a_2(x,y)x^2y^{\eta(l-2)}+\cdots+a_l(x,y)x^l$ be a power series in $P_{\eta^{l}}^{0}$. We write $a_i(0,0)$ for the constant term of $a_i\in k[[x,y]]$. Assume the polynomial $y^{l}+a_1(0,0)y^{l-1}+a_2(0,0)y^{l-2}+\cdots+a_l(0,0)$ has distinct roots. Then $\hat{g}$ factorizes uniquely as $\hat{g}=g_1g_2\cdots g_l$, where each $g_i\in P^0_{\eta}$.
    \end{enumerate}
\end{lemma}
\begin{proof}
    We start with Part (1). The uniqueness part follows from the fact that $k[[x,y]]$ is a UFD. We prove the existence part by induction on $t$. Let $s=\displaystyle\sum_{i=1}^t l_i$ be the length of $\boldsymbol\eta$. Let \[\hat{f}=y^{\gamma_0}+\sum_{i=1}^{s}a_i(x,y)x^iy^{\gamma_i}, \text{where } a_i(x,y)\in k[[x,y]]^{\times}.
    \]
    In order to show that $\hat{f}$ factorizes as required, it is enough to show that $f$ factorizes as $\hat{f}=gh$, where
    \[
        g=y^{\gamma_0-\eta_tl_t}+\sum_{i=1}^{s-l_t}b_i(x,y)x^iy^{\gamma_i-\eta_tl_t}\in P^0_{\eta_1^{l_1}\eta_2^{l_2}\cdots\eta_{t-1}^{l_{t-1}}}
    \]
    and
    \[
        h=y^{\eta_t l_t}+c_1(x,y)xy^{\eta_t(l_t-1)}+c_2(x,y)x^2y^{\eta_t(l_t-2)}+\cdots+c_l(x,y)x^{l_t}\in P^0_{\eta_t^{l_t}}.
    \]
   Comparing coefficients, we have 
    \[
    \begin{cases}
        b_1+c_1y^{\square}=a_1\\
        b_2+b_1c_1 y^{\square}+c_2 y^{\square}=a_2\\\cdots \\
        b_{s-l_t}+ b_{s-l_t-1}c_1y^{\square}+ b_{s-l_t-2}c_2y^{\square}+\cdots+b_{s-2l_t}c_{l_t}y^{\square}=a_{s-l_t}\\ b_{s-l_t}c_1+ b_{s-l_t-1}c_2y^{\square}+ b_{s-l_t-2}c_3y^{\square}+\cdots+b_{s-2l_t+1}c_{l_t}y^{\square}=a_{s-l_t+1}\\
        \cdots\\
        b_{s-l_t}c_{l_t-1}+ b_{s-l_t-1}c_{l_t}y^{\square}=a_{s-1}\\
        b_{s-l_t}c_{l_t}=a_s,        
    \end{cases}
    \]
    where $y^{\square}$ stands for raising $y$ to some positive integer power. Since $a_s$ is invertible $b_{s-l_t}$ is also invertible by the last equation. Solving this system of equations is equivalent to solving a single equation with variable $b_{s-l_t}$. Indeed, we can solve $c_{l_t}$, $c_{l_t-1}$, \dots, $c_1$ in turn as functions of $b_{s-l_t}$ from the last $l_t$ equations; then we can solve $b_1, b_2\dots, b_{s-l_t-1}$ in turn as functions of $b_{s-l_t}$ from the first $s-l_t-1$ equations; then we get the desired equation with variable $b_{s-l_t}$ by substituting the other variables as functions of $b_{s-l_t}$ in the $(s-l_t)$-th equation. 
This equation has a solution by Hensel's lemma. Indeed, after reduction to $k$, this equation has a unique solution $b_{s-l_t}(0,0)=a_{s-l_t}(0,0)$.

For Part (2), since we assume $y^{l}+a_1(0,0)y^{l-1}+a_2(0,0)y^{l-2}+\cdots+a_l(0,0)$ has distinct roots, by Hensel's lemma, they lift to distinct roots of the polynomial $$y^{l}+a_1(x,y)y^{l-1}+a_2(x,y)y^{l-2}+\cdots+a_l(x,y),$$ which we denote by $b_1(x,y), b_2(x,y), \dots, b_l(x,y)$. Then 
\[
\hat{g}=\prod_{i=1}^l(y^{\eta}-b_i(x,y)x)
\]
gives the desired factorization. Since $a_l(x,y)\in k[[x,y]]^{\times}$, we have $b_i(x,y)\in k[[x,y]]^{\times}$ for each $i$. 
\end{proof}

In order to obtain a spectral description of parabolic Hitchin fibers, we define the following open subset of the Hitchin base $B_P$.
\begin{definition}
We define $B_P^{0}$ to be the subset of $B_P$ such that $b\in B_P^{0}$ is characterized by the following properties:
	\begin{enumerate}[(2.1)]
		\item \label{open: first}$\Sigma_b\backslash q'$ is smooth,
		\item \label{open: second}$\hat{f}$ lies in $P_{\boldsymbol\lambda}^{0}$, and all components in the factorization of $\hat{f}$ in Lemma \ref{lemma: factor} Part (1) satisfy the assumption in Lemma \ref{lemma: factor} Part (2). It follows that $\hat{f}$ factorizes as $\hat{f}=f_1f_2\cdots f_r$, where $f_i=y^{\lambda_i}+a_i(x,y)x$, $a_i(x,y)\in k[[x,y]]^{\times}$. If $\lambda_s=\lambda_t$ for some $s\neq t$, then the constant terms of $a_s$ and $a_t$ are not equal to each other. 
	\end{enumerate}
\end{definition}
In particular, if $P$ is a Borel subgroup of $GL_n(k)$, $B_{P}^{0}$ is characterized by the spectral curve being smooth.

\begin{lemma}\label{lemma: the open subset}
    For every $b\in B_P^{0}$, there exists a $k$-point of $\mathcal{H}iggs_{n,P}$ that is mapped to $b$ under the Hitchin map $h_P$.
\end{lemma}
\begin{proof}
    Let $\widetilde{\Sigma}_b\longrightarrow \Sigma_b$ be the normalization of the spectral curve $\Sigma_b$ and let $\widetilde{\pi}:\widetilde{\Sigma}_b\longrightarrow X$ be the projection to $X$. Let $D=\on{Spec}\hat{\mathcal{O}}_{X,q}$ be the formal disk around $q$. By (2.\ref{open: second}), $\widetilde{\Sigma}_b\times_{X}D$ is the disjoint union of $\Sigma_i$, where
    \[\Sigma_i\cong \on{Spec}k[[x,y]]/(y^{\lambda_i}+a_i(x,y)x), a_i(x,y)\in k[[x,y]]^{\times}.
    \]
    Let $\mathcal{L}$ be an invertible sheaf on $\widetilde{\Sigma}_b$, then $\widetilde{\pi}_{*}(\mathcal{L})$ defines a $k$-point $(\widetilde{\pi}_{*}(\mathcal{L}), \phi)$ of $\mathcal{H}iggs_{n,q}$ such that $h(\widetilde{\pi}_{*}(\mathcal{L}),\phi)=b$ and $\on{res}_q\phi\in \mathcal{O}_P$ the Richardson orbit of $P$. Therefore we can find a partial flag structure on $\widetilde{\pi}_{*}(\mathcal{L})_q$ such that $\on{res}_q\phi$ is nilpotent with respect to this partial flag structure. 
\end{proof}

\begin{remark}\label{remark: Richardson orbit}
Let $(E, E_P,\tau,\phi)$ be a $k$-point of $\mathcal{H}iggs_{n,P}$ that is mapped to $b\in B_P^{0}$. Condition (2.\ref{open: second}) on $\Sigma_b$ enforces that $\on{res}_q(\phi)$ lies in the Richardson orbit $\mathcal{O}_P$. Note that $\mathcal{O}_P\bigcap n_P^{+}$ consists of a single $P$-orbit. Since we are in type $A$, for any $x\in \mathcal{O}_P\bigcap n_P^{+}$, the centralizer of $x$ in $GL_n(k)$ lies in $P$. Therefore there is a unique partial flag structure on $E_q$ that is compatible with $\on{res}_q(\phi)$.
\end{remark}

\begin{prop}\label{prop: the open subset is non=empty}
    In the following two cases:
     \begin{enumerate}
         \item $g_X\geq 2$,
         \item $g_X=1$, $n\geq 3$ and $P\subseteq G$ is a Borel subgroup,
     \end{enumerate}
     $B_P^{0}$ is Zariski open dense in $B_P$. Moreover, $B_P$ is the scheme-theoretic image of the Hitchin map $h_P$, i.e. the smallest closed subscheme of $B$ through which $h_P$ factors. 
\end{prop}

\begin{proof}
    The first statement together with Lemma \ref{lemma: the open subset} implies the second statement. For the first statement, we only need to show that both (2.\ref{open: first}) and (2.\ref{open: second}) define a non-empty open subset in $B_P$.

We start by showing that (2.\ref{open: first}) defines a non-empty open subset in $B_P$. We denote by $B_P^{sm}$ the locus in $B_P$ where the spectral curves are smooth away from $q'$. Since $B_P^{sm}\subset B_P$ is open, it is enough to show that it is non-empty. 

\emph{Case 1.} $g_X\geq 2$, except for the case when $g_X=2$, $n=2$, $P=GL_2(k)$. We use the following version of Bertini's theorem in \cite{CGM}:

\begin{theorem}[cf. \cite{CGM}, Corollary 1] \label{thm: bertini}
    Let $V$ be a smooth algebraic variety over an algebraically closed field $k$. Let $S$ be a finite-dimensional linear system on $V$. Assume that the rational map $V\dashrightarrow P^{N}$ corresponding to $S$ induces (whenever defined) separably generated residue field extensions. Then a generic element of $S$ defines a subscheme of $V$ that is smooth away from the base locus of $S$. 
\end{theorem}

Let $\pi$ be the projection $\pi: T^{*}X(q)\longrightarrow X$. We denote by $y$ the tautological section of $\pi^{*}(\Omega_X(q))$. Let $S$ be the linear system of sections in $\pi^{*}(\Omega_X(q)^n)$ spanned by $y^n$ and $\pi^{*}(b_i)y^{n-i}$ for all
$b_i\in \Gamma(X, \Omega_X^{\otimes i}((i-m_i)q)), i=1,2,\dots,n$. 
The section $y^n$ is not contained in the span of $\pi^{*}(b_i)y^{n-i}$. The set of spectral curves $\Sigma_b$ with $b\in B_P$ corresponds to the open subset of $S$ defined by the coefficient of $y^n$ being non-zero. Let $
N=\on{dim}(S)-1$. We denote by $f_S: T^{*}X(q)\dashrightarrow P^{N}$ the map induced by $S$. 
In order to apply Theorem \ref{thm: bertini}, we show that $f_S$ is unramified away from $\pi^{-1}(q)$, which will imply that $f_S$ induces finite separable extensions on the residue fields when restricted to $T^{*}(X\backslash q) $. By the exact sequence 
\[
    f_S^{*}\Omega_{P^N}|_{T^{*}(X\backslash q)}\xrightarrow{\nu} \Omega_{T^{*}(X\backslash q)}\longrightarrow \Omega_{T^{*}(X\backslash q)/P^{N}}\longrightarrow 0,
\]
it is enough to show that for any $k$-point $p'$ on $T^{*}X(q)$ such that $\pi(p')=p\neq q$, the map $\nu$ induces a surjection onto the fiber of $\Omega_{T^{*}(X\backslash q)}$ at $p'$.

Let $V=\on{Spec}(A)$ be an affine open neighborhood of $p$ in $X$. Let $x$ be an element of $A$ that is mapped to a local parameter of $X$ at $p$. Shrinking $V$ if necessary, we assume $q\notin V$ and $dx$ is a nowhere vanishing section of $\Omega^1_V$. Let $U=\pi^{-1}(V)$. The section $dx$ gives a trivialization of $T^{*}X(q)|_V$ and $\pi^{*}(T^{*}X(q))|_U$. Under this trivialization, the tautological section $y$ is equal to $\partial_x$ considered as an element in $\mathcal{O}_{U}$, and $\Omega_{U}$ is a free $\mathcal{O}_U$-module generated by $dx$ and $dy$. The fiber of $\Omega_{U}$ at $p'$ is a $k$-vector space of dimension two spanned by $dx$ and $dy$. 

Under our assumptions on $g_X$, $n$ and $P$, we have 
\begin{align*}
    &\on{dim}_k\Gamma(X, \Omega_X^{\otimes i}((i-m_i)q))-\on{dim}_k\Gamma(X, \Omega_X^{\otimes i}((i-m_i)q-p))
    =1, \text{for } i\geq 1, \\
    &\on{dim}_k\Gamma(X, \Omega_X^{\otimes i}((i-m_i)q-p))-\on{dim}_k\Gamma(X, \Omega_X^{\otimes i}((i-m_i)q-2p))=1, \text{for } i\geq 2.
\end{align*}
    Take
\begin{align*}
    &s_1\in \Gamma(X, \Omega_X^{\otimes n}((n-m_n)q))\backslash\Gamma(X, \Omega_X^{\otimes n}((n-m_n)q-p)),\\
    &s_2\in \Gamma(X, \Omega_X^{\otimes n}((n-m_n)q-p))\backslash\Gamma(X, \Omega_X^{\otimes n}((n-m_n)q-2p)),\\
    &s_3 \in \Gamma(X, \Omega_X^{\otimes (n-1)}((n-1-m_{n-1})q))\backslash\Gamma(X, \Omega_X^{\otimes (n-1)}((n-1-m_{n-1})q-p)).
\end{align*}
then $d(s_2/s_1)$ and $d(s_3y/s_1)$ span the fiber of $\Omega_{U}$ at $p'$.

Now we apply Theorem \ref{thm: bertini} to the restriction of the linear system $S$ to $T^{*}(X\backslash q)$. Since $q'\in \pi^{-1}(q)$ is the only base point of $S$, a spectral curve $\Sigma_b$ is smooth away from $q'$ for a generic $b\in B_P$. 

\emph{Case 2.} $g_X=2$, $n=2$, $P=GL_2(k)$.
By the same arguments as in \emph{Case 1}, the map $f_S: T^{*}(X\backslash q)\longrightarrow P^{N}$ is unramified away from the union of $\pi^{-1}(p)$ for all $p\in X\backslash q$ that satisfies $\mathcal{O}(2p)\cong \Omega_X$. There are finite many points of $X$ with this property, therefore the fact that a generic spectral curve is smooth away from $q'$ follows from the following lemma:
\begin{lemma}
    Let $p\in X\backslash q$. For a generic $b\in B_P$, the spectral cover $\Sigma_b\longrightarrow X$ is \'etale around $p$. 
\end{lemma}
\begin{proof}
    This follows easily from the calculation
    \[\on{dim}_k\Gamma(X, \Omega_X^{\otimes i}((i-m_i)q))-\on{dim}_k\Gamma(X, \Omega_X^{\otimes i}((i-m_i)q-p))
    =1, \text{for } i\geq 1.
    \]
\end{proof}

\emph{Case 3.} $g_X=1$. We consider the subspace $\displaystyle\bigoplus_{i=1}^n\Gamma(X,\Omega_X^{\otimes i})\subseteq B_P$. Since $\Omega_X$ is isomorphic to $\mathcal{O}_X$, it is easy to find $b\in \displaystyle\bigoplus_{i=1}^n\Gamma(X,\Omega_X^{\otimes i})$ such that the spectral cover $\Sigma_b\longrightarrow X$ is \'etale away from $\pi^{-1}(q)$. 

Now we turn to (2.\ref{open: second}). Let $b\in B_P$. The condition $\hat{f_b}\in P_{\boldsymbol\lambda}^{0}$ is equivalent to the condition that for $i=1,2,\dots,r$, the $(n-\gamma_i)$-th component of $b$ lies in  
\[
\Gamma(X, \Omega_X^{\otimes (n-\gamma_i)}((n-\gamma_i-i)q))\backslash\Gamma(X, \Omega_X^{\otimes (n-\gamma_i)}((n-\gamma_i-i-1)q)).
\]
This condition defines a non-empty open subset of $B_P$ since
\[
    \on{dim}_k\Gamma(X, \Omega_X^{\otimes (n-\gamma_i)}((n-\gamma_i-i)q))-\on{dim}_k\Gamma(X, \Omega_X^{\otimes (n-\gamma_i)}((n-\gamma_i-i-1)q))=1
\]
under our assumptions on $g_X$, $n$ and $P$.
The fact that the second condition in (2.\ref{open: second}) defines a non-empty open subset follows easily from the uniqueness part of Lemma \ref{lemma: factor}.
\end{proof}

\begin{remark}
The second statement in Proposition \ref{prop: the open subset is non=empty} was previously obtained in \cite{BK18} using different methods. 
\end{remark}

\begin{remark}\label{Remark: P1}
    Proposition \ref{prop: the open subset is non=empty} also holds for the case of $X=\textbf{P}^1$ with ramification at $D=q_1+q_2+\cdots+q_m$, $m\geq 4$ and each parabolic subgroup $P_i$ is a Borel subgroup. We need to show that for a generic $b\in B_{P_D}=\displaystyle\bigoplus_{i=1}^n \Gamma(\textbf{P}^1, \Omega_{\textbf{P}^1}^{\otimes i}((i-1)D))$, the spectral curve $\Sigma_b$ is smooth. For each $i=1,2,\dots,m$, a generic spectral curve is smooth above $q_i$ since \[\on{dim}_k\Gamma(\textbf{P}^1, \Omega_{\textbf{P}^1}^{\otimes n}((n-1)D))-\on{dim}_k\Gamma(\textbf{P}^1, \Omega_{\textbf{P}^1}^{\otimes n}((n-1)D-q_i))=1.
    \]
    Therefore it is enough to show that there exists $b\in B_{P_D}(k)$ such that $\Sigma_b$ is smooth away from $\pi^{-1}(q_i)$. If $n\geq 3$, the same arguments as in \emph{Case 1} of the proof of Proposition \ref{prop: the open subset is non=empty} would work. If $n=2$, we consider the subspace \[\Gamma(\textbf{P}^1,\Omega_{\textbf{P}^1})\oplus \Gamma(\textbf{P}^1,\Omega^{\otimes 2}_{\textbf{P}^1}(q_1+q_2+q_3+q_4))\subseteq B_{P_D}.
    \]
    Since $\Gamma(\textbf{P}^1,\Omega_{\textbf{P}^1})=0$ and $\Gamma(\textbf{P}^1,\Omega^{\otimes 2}_{\textbf{P}^1}(q_1+q_2+q_3+q_4))\cong\Gamma(\textbf{P}^1, \mathcal{O}_{\textbf{P}^1})=k$, the spectral curve $\Sigma_b$ is \'etale away from $\pi^{-1}(q_i)$ for any $b\in k^{\times}$.
\end{remark}

\subsection{Spectral data of parabolic Higgs bundles}\label{subsection: spectral data}
The next theorem describes the spectral data of parabolic Higgs bundles. 
\begin{theorem}[cf. \cite {SS}, Theorem 5.16]\label{thm: spectral data single}
For $b\in B_P^{0}(k)$, the fiber of the Hitchin map $h_P^{-1}(b)$ is isomorphic to the Picard stack $\on{Pic}(\widetilde{\Sigma}_b)$. Here $\sigma: \widetilde{\Sigma}_b\longrightarrow\Sigma_{b}$ is the normalization of the spectral curve $\Sigma_{b}$.
\end{theorem}
\begin{proof}
    We've already constructed a map $\on{Pic}(\widetilde{\Sigma}_b)\longrightarrow h_P^{-1}(b)$ in the proof of Lemma \ref{lemma: the open subset}, therefore it is enough to construct the inverse map. Let $(E,\phi)\in h_P^{-1}(b)$, and we denote by $\mathcal{F}\in{\on{Coh}}(\Sigma_b)$ the corresponding spectral sheaf. Our goal is to show that there is a natural sheaf $\mathcal{L}\in{\on{Coh}}(\widetilde{\Sigma}_b)$ such that $\sigma_{*}(\mathcal{L})=\mathcal{F}$.
    
   Let $\hat{\Sigma}_b=\on{Spec}\hat{\mathcal{O}}_{\Sigma_b,q'}$. We write $\boldsymbol\lambda=\lambda_1^{l_1}\lambda_2^{l_2}\cdots \lambda_t^{l_t}$, $\lambda_1\geq \lambda_2\geq\cdots\geq \lambda_t$. Note that by condition (2.\ref{open: second}) in the definition of $B_P^0$,
    \[  \mathcal{O}_{\hat{\Sigma}_b}\cong k[[x,y]]/(\hat{f}), \text{ and } \hat{f}=\prod_{i=1}^{t}\prod_{j=1}^{l_i}(y^{\lambda_i}-a_{ij}x), a_{ij}\in k[[x,y]]^{\times}.
    \]
    Therefore
    \[       \widetilde{\Sigma}_b\times_{\Sigma_b}\hat{\Sigma}\cong \coprod_{i=1}^{t}\coprod_{j=1}^{l_i}{\Sigma_{ij}}, \text { where }\mathcal{O}_{\Sigma_{ij}}\cong k[[x,y]]/(y^{\lambda_i}-a_{ij}x). 
    \]
    Each $\Sigma_{ij}$ is a formal disk such that the closed point is mapped to $q'$ under $\sigma: \widetilde{\Sigma}_b\to \Sigma_b$.

    Note that since the action of $y$ on $\mathcal{F}/x\mathcal{F}$ as a matrix with Jordan blocks of type $\boldsymbol\lambda$, the element $v_1\coloneqq y^{\lambda_1}/x$ acts on the spectral sheaf $\mathcal{F}$ sheafifies over ${\hat{\Sigma}_b}^1$ defined by 
    \[
    \mathcal{O}_{{\hat{\Sigma}_b}^1}=k[[x,y]][v_1]/(\hat{f}, y^{\lambda_1}-xv_1). 
    \]
    This new curve ${\hat{\Sigma}_b}^1$ is a disjoint union of $l_1+1$ components 
    \begin{equation}\label{eq: decomposition of spectral curves through nil orbits}
        {\hat{\Sigma}_b}^1=\coprod_{j=1}^{l_1}    \on{Spec}k[[v_1,y]]/(v_1-a_{1j})\coprod\on{Spec}k[[x,y]][v_1]/(\prod_{i=2}^{t}\prod_{j=1}^{l_i}(y^{\lambda_i}-a_{ij}x), y^{\lambda_1}-xv_1). 
     \end{equation}
The first $l_1$ components are formal disks that correspond to $\Sigma_{1j}, j=1,2,\dots,l_1$ in the normalization curve. The spectral sheaf over those components must be line bundles, and each contributes a Jordan block of size $\lambda_1$ to the residue of the Higgs field at the marked point $q$. Let $\mathcal{F}_1$ be the spectral sheaf over the last component of ${\hat{\Sigma}_b}^1$ in 
(\ref{eq: decomposition of spectral curves through nil orbits}). Since $y$ acts on $\mathcal{F}_1/x\mathcal{F}_1$ as a matrix with Jordan blocks of type $\lambda_2^{l_2}\lambda_3^{l_3}\cdots \lambda_t^{l_t}$, the element $v_2\coloneqq y^{\lambda_2}/x$ acts on $\mathcal{F}_1$, therefore $\mathcal{F}_1$ sheafifies over ${\hat{\Sigma}_b}^2$ defined by 
 \[
    \mathcal{O}_{{\hat{\Sigma}_b}^2}=k[[x,y]][v_2]/(\prod_{i=2}^{t}\prod_{j=1}^{l_i}(y^{\lambda_i}-a_{ij}x), y^{\lambda_2}-xv_2). 
    \]
Repeating the same procedure for $t$ times, the spectral sheaf $\mathcal{F}$ over ${\hat{\Sigma}_b}$ that we start with decomposes as
\[
\mathcal{F}=\bigoplus_{i=1}^{t}\bigoplus_{j=1}^{l_i}\mathcal{L}_{{ij}},
\]
where each $\mathcal{L}_{{ij}}$ is a line bundle over $\Sigma_{ij}$. Since the normalization curve $\widetilde{\Sigma}_b$ locally is the disjoint union of those $\Sigma_{ij}$, we get the desired statement that the spectral sheaf $\mathcal{F}$ sheafifies over $\widetilde{\Sigma}_b$.
\end{proof}

For the purpose of this paper, we need to develop a family version of Theorem \ref{thm: spectral data single}. The first step is to construct a simultaneous normalization of the family of spectral curves above $B_P^{0}$. This can be done since the spectral curves above $B_P^{0}$ are equisingular. To be more precise, let $\Sigma\subseteq B_P^{0}\times T^{*}X(q)$ be the global spectral curve above $B_P^{0}$; we will construct a new family of curves $\widetilde{\Sigma}\longrightarrow B_P^{0}$ with a proper birational morphism $\sigma: \widetilde{\Sigma}\longrightarrow \Sigma$ such that for each $b\in B_{P}^{0}(k)$, the morphism $\sigma_b: \widetilde{\Sigma}_b\longrightarrow \Sigma_b$ is the normalization of $\Sigma_b$.

The construction is as follows. Recall that $q'$ is the closed point of $T^{*}X(q)$ above $q\in X$ that lies in the zero section of $T^{*}X(q)$. We blow up $B_P^{0}\times T^{*}X(q)$ along $B_P^{0}\times q'$, and denote the strict transform of $\Sigma$ by $\Sigma^1$. Let $V$ be an open neighborhood of $q$ and $U=\pi^{-1}(V)$. For $b\in B_P^{0}(k)$, let $\Sigma_b(V)\coloneqq V \times _X \Sigma_b$, then $\mathcal{O}_{\Sigma_b(V)}$ is isomorphic to $\mathcal{O}_U/(f)$ for some  $f=y^n+b_{1}y^{n-1}+\dots + b_{n-1}y+b_n$, $b_i\in \mathcal{O}_V$. Since $b\in B_P^{0}(k)$, $\hat{f}$ factorizes as $\hat{f}=f_1f_2\cdots f_r$, $f_i\in P^0_{\lambda_i}$. We write
\[f_i=y^{\lambda_i} + a_{i}(x,y)x, \text{ where } a_i(x,y)\in k[[x,y]]^{\times}.
\]
We denote $V\times _X \Sigma^1_b$ by $\Sigma^1_b(V)$, then $\Sigma^1_b(V)$ is a closed subvariety of \[\on{Spec}(\mathcal{O}_U[u]/(x-yu)).
\]
We denote by $q'_1$ the point defined by $y=u=0$. By assumption (2.\ref{open: first}) and the second part of assumption (2.\ref{open: second}) in the definition of $B_P^0$, $\Sigma^1_b$ is smooth away from $q'_1$. Let $\hat{\Sigma}^1_b\cong \on{Spec}(\hat{\mathcal{O}}_{\Sigma_b^1, q'_1})$, then
\[
\mathcal{O}_{\hat{\Sigma}^1_b}\cong k[[u,y]]/(\prod_{i=1}^{t}(y^{\lambda_i-1}+a_i(yu,y)u)),
\]
where $t$ is the largest integer so that $\lambda_t-1 > 0$. Let
\[
    g=\prod_{i=1}^{t}(y^{\lambda_i-1}+a_i(yu,y)u) \text{ and } g_i=y^{\lambda_i-1}+a_i(yu,y)u,
\]
so $g$ factorizes as $g=g_1g_2\cdots g_t$. In each $g_i$, there is a unique monomial of the form $y^m$, and the degree of such monomial is in decreasing order. Compared to $f_1$, the degree of such monomial in $g_1$ is lower by $1$. This observation guarantees that the family of curves $\Sigma$ can be resolved simultaneously by $\lambda_1$ steps of blow-ups. 
Now we blow up $\on{Spec}(\mathcal{O}_U[u]/(x-yu))$ along $B_P^{0}\times q'_1$, and denote the strict transform of $\Sigma^1$ by $\Sigma^2$. Repeating this procedure, we get a series of families of curves above $B_P^{0}$:
\[
    \Sigma^{\lambda_1}\longrightarrow \Sigma^{\lambda_1-1}\longrightarrow \cdots \Sigma^1\longrightarrow \Sigma.
\]
It follows from the observation above that $\Sigma^{\lambda_1}_b$ is smooth for each $b\in B_P^{0}(k)$. The morphism $\Sigma^{\lambda_1}\longrightarrow B_P^{0}$ is flat since each $\Sigma^{\lambda_1}_b$ is a projective curve of the same genus. The morphism $\Sigma^{\lambda_1}\longrightarrow \Sigma$ is proper and birational by properties of strict transforms. We set $\widetilde{\Sigma}\cong \Sigma^{\lambda_1}$. 

\begin{remark}
    After our paper appeared on the arXiv, similar results as in Theorem \ref{thm: spectral data single} were also obtained in \cite{SWW}, see Theorem 1.1. In \cite{SWW} the authors also considered the generic fiber of so-called weak parabolic fibrations, in which the residue of the Higgs field is not required to be nilpotent. We will prove Theorem \ref{thm: spectral data family} the family version of Theorem \ref{thm: spectral data single}, which did not appear in \cite{SWW}. 
\end{remark}

Now we are ready to state the following theorem, which is a family version of Theorem \ref{thm: spectral data single}. We denote $\mathcal{H}iggs_{n,P}\times_{B_{P}}B^{0}_P$ by $\mathcal{H}iggs^{0}_{n,P}$.

\begin{theorem}\label{thm: spectral data family}
    The correspondence between Higgs bundles and spectral sheaves induces an isomorphism of stacks over $B_P^{0}$:
    \[
        \mathcal{H}iggs^{0}_{n,P}\cong \on{Pic}(\widetilde{\Sigma}/B_P^{0}).
    \]
    
\end{theorem}

\begin{proof}
    Let $S$ be a $k$-scheme. Since both $\mathcal{H}iggs^{0}_{n,P}$ and $B_P^0$ are locally of finite type over $k$, we can assume $S$ is locally of finite type over $k$. Let $(E,\phi)$ be an $S$-point of $\mathcal{H}iggs^{0}_{n,P}$ such that $h(E,\phi)=b\in B_P^{0}(S)$.
    We denote by $\mathcal{F}$ the corresponding spectral sheaf on $\Sigma_b$. The goal is to construct a sheaf $\widetilde{\mathcal{F}}$ on $\widetilde{\Sigma}_b$ such that $(\sigma_b)_{*}\widetilde{\mathcal{F}}=\mathcal{F}$. We set $\Sigma^{0}_b=\Sigma_b$, $\mathcal{F}_0=\mathcal{F}$. The strategy is to construct by induction a series of sheaves $\mathcal{F}_k$ on $\Sigma^{k}_b$, $k=1,2,\dots \lambda_1$, such that $(p_k)_{*}\mathcal{F}_k=\mathcal{F}_{k-1}$, where $p_k$ is the map $p_k: \Sigma_b^{k}\longrightarrow \Sigma_b^{k-1}$. We assume that we already have $\mathcal{F}_0, \mathcal{F}_1,\dots,\mathcal{F}_{t-1}$ with the required property and aim to obtain $\mathcal{F}_t$. Note that above $V$ an open neighborhood of $q$, while obtaining $\Sigma_b^{k}$, we add a new variable $u_k$ to $\mathcal{O}_{\Sigma_b^{k-1}}$ and impose $u_{k-1}=u_ky$, starting from $u_0=x$. Therefore in order to construct $\mathcal{F}_t$ so that $(p_t)_{*}\mathcal{F}_t=\mathcal{F}_{t-1}$, all we need to do is to define an action of $u_{t-1}/y$ on $\mathcal{F}_{t-1}$. Note that for any $s: \on{Spec}(k)\longrightarrow S$ a closed point of $S$, $s^{*}\mathcal{F}_{t-1}$ is a torsion-free sheaf on $(\Sigma_{b}^{t-1})_{s}\coloneqq \Sigma_{b}^{t-1}\times_{S,s} \on{Spec}(k)$, therefore if such an action exists, it is unique. For the existence of such an action, we consider the coherent sheaf $\mathcal{G}=u_{t-1}\mathcal{F}_{t-1}/u_{t-1}\mathcal{F}_{t-1}\bigcap y\mathcal{F}_{t-1}$ on $\Sigma_b^{t-1}$. There exists an action of $u_{t-1}/y$ on $\mathcal{F}_{t-1}$ if and only if $\mathcal{G}=0$. By Theorem \ref{thm: spectral data single}, such an action exists when restricted to $s$, so $s^{*}\mathcal{G}=0$ for all closed points $s$ of $S$. Therefore $\mathcal{G}=0$. 
    
    We set $\widetilde{\mathcal{F}}=\mathcal{F}_{\lambda_1}$. Since $\widetilde{\Sigma}_b$ is smooth, $\widetilde{\mathcal{F}}$ is an invertible sheaf. Now let $(E_1,\phi_1)$ and $(E_2,\phi_2)$ be two $S$-points of $\mathcal{H}iggs^{0}_{n,P}$, both mapped to $b$ under the Hitchin map, and we denote the corresponding spectral sheaves by $\mathcal{F}_1$ and $\mathcal{F}_2$. The construction of $\widetilde{\mathcal{F}}$ implies that there is an isomorphism $\mathcal{H}om_{\mathcal{O}_{\Sigma_b}}(\mathcal{F}_1,\mathcal{F}_2)\cong \mathcal{H}om_{\mathcal{O}_{\widetilde{\Sigma}_b}}(\widetilde{\mathcal{F}}_1,\widetilde{\mathcal{F}}_2)$. Therefore we have a morphism of stacks $\mathcal{H}iggs^{0}_{n,P}\longrightarrow \on{Pic}(\widetilde{\Sigma}^{0}/B_P^{0})$.
    
    The inverse of this morphism is constructed as follows. Let $\mathcal{L}$ be an invertible sheaf on $\widetilde{\Sigma}_b$. Since $\Sigma_b\subseteq S\times T^{*}X(q)$, there is a morphism
    \[
        \mathcal{O}_S\boxtimes\pi^{*}\mathcal{T}_X(-q)\longrightarrow \mathcal{E}nd_{\mathcal{O}_{\Sigma_b}}(\sigma_{*}\mathcal{L}).
    \]
    By adjunction, we get a morphism
    \[
    \mathcal{O}_S\boxtimes \mathcal{T}_X(-q)\longrightarrow \pi_{*}{\mathcal{E}nd}_{\mathcal{O}_{\Sigma_b}}(\sigma_{*}\mathcal{L})\longrightarrow \mathcal{E}nd_{\mathcal{O}_X}(\widetilde{\pi}_{*}\mathcal{L}).
    \]
    By Remark \ref{remark: Richardson orbit}, there is a unique parabolic reduction of $\widetilde{\pi}_{*}\mathcal{L}$ at $q$ that is compatible with this Higgs field. 
\end{proof} 

\section{Azumaya property of differential operators in positive characteristic}
\label{section: Azumaya property of differential operators }
\subsection{Frobenius twist of a $k$-scheme}\label{Secion: Frobenius twist}
Let $Y$ be a scheme over an algebraically closed field $k$ of characteristic $p$. Recall that the absolute Frobenius $F_Y: Y\longrightarrow Y$ is the map that fixes the underlying topological space and takes $f$ to $f^p$ on regular functions. The Frobenius twist $Y^{(1)}$ of $Y$ is the $k$-scheme that fits into the following pull-back diagram:
\bd
\xymatrix{
Y^{(1)} \ar[r] \ar[d] & Y \ar[d] \\
\on{Spec} k \ar[r]^{F_{\on{Spec} k}} & \on{Spec} k
}
\ed
The relative Frobenius $\on{Fr}:Y\longrightarrow Y^{(1)}$ is the unique map that makes the following diagram commute.
\bd
\xymatrix{
Y\ar@/_/[ddr] \ar@/^/[drr]^{F_Y} \ar@{.>}[dr]|-{\exists \on{Fr}}\\
&Y^{(1)} \ar[d] \ar[r] & Y\ar[d]\\
& \on{Spec} k \ar[r]^{F_{\on{Spec} k}} & \on{Spec} k}
\ed
Since $\on{Fr}$ induces a bijection on $k$-points, we will not distinguish between $k$-points on $Y$ and $k$-points on $Y^{(1)}$. 

Let $\mathcal{F}$ and $\mathcal{G}$ be two $\mathcal{O}_Y$-modules. A map $\varphi: \mathcal{F}\longrightarrow\mathcal{G}$ is called $p$-linear if it is additive and satisfies $\varphi(fs)=f^{p}\varphi(s)$ for any $f\in\mathcal{O}_U$, $s\in\mathcal{F}(U)$ and open $U\subseteq Y$. For any $\mathcal{O}_Y$-module $\mathcal{F}$, there is a natural $p$-linear map $\mathcal{F}\longrightarrow (F_Y)^{*}\mathcal{F}$. This map is ``universally $p$-linear'' in the sense that any $p$-linear map $\mathcal{F}\longrightarrow \mathcal{G}$ factors through $\mathcal{F}\longrightarrow (F_Y)^{*}\mathcal{F}$ and gives a unique $\mathcal{O}_Y$-linear map $(F_Y)^{*}\mathcal{F}\longrightarrow \mathcal{G}$.

\subsection{Azumaya property of differential operators}\label{subsection: Azumaya}
In this section we review the Azumaya property of crystalline differential operators in characteristic $p$, following \cite{BB}. Let $Y$ be a smooth variety over $k$. We denote by $D_Y$ the sheaf of crystalline differential operators on $Y$, i.e. the sheaf of algebras generated by $\mathcal{O}_Y$ and $\mathcal{T}_Y$ subject to the relations: $\partial f-f\partial=\partial(f)$, $\partial_1\partial_2-\partial_2\partial_1=[\partial_1,\partial_2]$, for any $f\in \mathcal{O}_U$, $\partial, \partial_1,\partial_2\in \mathcal{T}_Y(U)$ and $U\subseteq Y$ open. Since we are in characteristic $p$, for any $\partial\in \mathcal{T}_Y(U)$, $\partial^p\in D_Y$ acts as a derivation on $\mathcal{O}_U$, and we denote this derivation by $\partial^{[p]}\in \mathcal{T}_Y(U)$. There is a $p$-linear map $\mathcal{T}_Y\longrightarrow D_Y$ defined by $\iota(\partial)=\partial^p-\partial^{[p]}$. By the discussion above, $\iota$ induces an $\mathcal{O}_Y$-linear map $\on{Fr}^{*}\mathcal{T}_{Y^{(1)}}\cong F_Y^{*}\mathcal{T}_Y\longrightarrow D_Y$. By adjunction, we have an $\mathcal{O}_{Y^{(1)}}$-linear map
\[\iota: \mathcal{T}_{Y^{(1)}}\longrightarrow \on{Fr}_{*}D_Y.
\]
Therefore $\on{Fr}_{*}D_Y$ sheafifies on $T^{*}Y^{(1)}$, i.e. there exists a sheaf of algebras $\mathcal{D}_Y$ on $T^{*}Y^{(1)}$ that satisfies $\pi^{(1)}_{*}\mathcal{D}_Y\cong \on{Fr}_{*}D_Y$.

The following theorem is proved in \cite{BMR}. 

\begin{theorem}[cf. \cite{BB} Theorem 3.3 and \cite{BMR} Theorem 2.2.3]\label{thm: Azumaya}
	\mbox{}
	\begin{enumerate}
		\item The map $\iota$ induces an isomorphism of sheaves from $\mathcal{O}_{T^{*}Y^{(1)}}$ to the center of $\mathcal{D}_Y$.
		\item The sheaf of algebras $\mathcal{D}_Y$ is an Azumaya algebra over $T^{*}Y^{(1)}$ of rank $p^{2d}$, where $d$ is the dimension of $Y$.  
	\end{enumerate}
\end{theorem}

Let $\mathcal{A}$ be an Azumaya algebra on $Y$. A splitting of $\mathcal{A}$ is defined to be a pair $(E,\rho)$, where $E$ is a locally free sheaf on $Y$ and $\rho: \mathcal{A}\xrightarrow{\simeq} \mathcal{E}nd(E)$ is an isomorphism of $\mathcal{O}_Y$-algebras. Such a $(E,\rho)$ induces an equivalence between the category $\on{QCoh}(Y)$ of quasi-coherent sheaves on $Y$ and the category $\mathcal{A}\on{-mod}$ of $\mathcal{A}$-modules, which maps $\mathcal{F}\in \on{QCoh}(Y)$ to $E\otimes \mathcal{F}$. We define an equivalence from an Azumaya algebra $\mathcal{A}$ to another Azumaya algebra $\mathcal{B}$ to be a splitting of $\mathcal{A}^{op}\otimes \mathcal{B}$. Such a splitting induces an equivalence from the category of $\mathcal{A}$-modules to the category of $\mathcal{B}$-modules. Note that if there is a locally free sheaf $E$ that gives a splitting of $\mathcal{A}^{op}\otimes \mathcal{B}$, then $\mathcal{H}om_{\mathcal{O}_Y}(E, \mathcal{O}_Y)$ gives a splitting of $\mathcal{A}\otimes \mathcal{B}^{op}$. 

Let $f: Z\longrightarrow Y$ be a morphism between smooth $k$-varieties. We denote by $df^{(1)}$ the Frobenius twist of the map induced by the differential of $f$:
\[
    df^{(1)}: Z^{(1)}\times_{Y^{(1)}} T^{*}Y^{(1)}\longrightarrow T^{*}Z^{(1)}.
\]
Let $p_2$ be the projection to $T^{*}Y^{(1)}$. Then we have:
\begin{prop}[cf. \cite{BB} Proposition 3.7]\label{Prop: pull back Azumaya}
    The Azumaya algebras $(df^{(1)})^{*}\mathcal{D}_Z$ and $p_2^{*}\mathcal{D}_Y$ are canonically equivalent.
\end{prop}

Following \cite{BB}, we define $f^{!}: \mathcal{D}_Y{\on{-mod}}\longrightarrow \mathcal{D}_Z{\on{-mod}}$ to be the composition of the pull-back functor $\mathcal{D}_Y{\on{-mod}}\longrightarrow p_2^*\mathcal{D}_Y{\on{-mod}}$, the equivalence in Proposition \ref{Prop: pull back Azumaya}, and the push-forward functor $df^{*}\mathcal{D}_Z{\on{-mod}}\longrightarrow \mathcal{D}_Z{\on{-mod}}$. Similarly, we define $f_{*}: \mathcal{D}_Z{\on{-mod}}\longrightarrow \mathcal{D}_Y{\on{-mod}}$ to be the composition of the pull-back functor $\mathcal{D}_Z{\on{-mod}}\longrightarrow df^*\mathcal{D}_z{\on{-mod}}$, the equivalence in Proposition \ref{Prop: pull back Azumaya}, and the push-forward functor $p_2^{*}\mathcal{D}_Y{\on{-mod}}\longrightarrow \mathcal{D}_Y{\on{-mod}}$.

Let $\theta_Y^{(1)}$ be the tautological 1-form on $T^{*}Y^{(1)}$. We think of $\theta_Y^{(1)}$ as a map:
\[
    \theta_Y^{(1)}: T^{*}Y^{(1)}\longrightarrow T^{*}(T^{*}Y)^{(1)}. 
\]
\begin{corollary}[cf. \cite{BB} Proposition 3.11 and Corollary 3.12]\label{cor: 1-form}
\mbox{}
\begin{enumerate}
    \item The Azumaya algebra $(\theta_Y^{(1)})^{*}\mathcal{D}_{T^{*}Y}$ is canonically equivalent to $\mathcal{D}_Y$.
    \item Let $\theta^{(1)}_1\in \Gamma(Z^{(1)}, \Omega_{Z^{(1)}})$ and $\theta^{(1)}_2\in \Gamma(Y^{(1)}, \Omega_{Y^{(1)}})$. If $(f^{(1)})^{*}(\theta_2^{(1)})=\theta_1^{(1)}$, then the two Azumaya algebras $(\theta_1^{(1)})^{*}\mathcal{D}_Z$ and $(\theta^{(1)}_2\circ f^{(1)})^{*}\mathcal{D}_Y$ are canonically equivalent. 
\end{enumerate} 
\end{corollary}

Let $\mathcal{M}$ be a $D_Y$-module. We denote by $\nabla_\mathcal{M}$ the corresponding flat connection $\mathcal{M}\longrightarrow \mathcal{M}\otimes \Omega_Y$. There is a $p$-linear map $\mathcal{T}_Y\longrightarrow \mathcal{E}nd(\mathcal{M})$ defined by \[\psi_{\nabla_\mathcal{M}}(\partial)=\nabla_\mathcal{M}(\partial)^p-\nabla_\mathcal{M}(\partial^{[p]}).
\]
By the discussion in Section \ref{Secion: Frobenius twist}, we can associate with it a $\mathcal{O}_Y$-linear map
\[
    \psi_{\nabla_{\mathcal{M}}}: \on{Fr}^{*}\mathcal{T}_{Y^{(1)}}\longrightarrow \mathcal{E}nd(\mathcal{M}),
\]
which we call the $p$-curvature of $\mathcal{M}$.

We review the Cartier descent for flat connections with zero $p$-curvature. Let $\mathcal{F}$ be a quasi-coherent sheaf on $Y^{(1)}$. There is a canonical $D_Y$-action on $\on{Fr}^{*}(\mathcal{F})\cong \mathcal{O}_Y\otimes_{\mathcal{O}_{Y^{(1)}}} \mathcal{F}$, which comes from the canonical action of $D_Y$ on $\mathcal{O}_Y$. Therefore we have a flat connection $(\on{Fr}^{*}\mathcal{F},\nabla_{\on{can}})$. This construction induces a functor from the category of quasi-coherent sheaves on $Y^{(1)}$ to the category of $D_Y$-modules on $Y$ with zero $p$-curvature.

\begin{theorem}[Cartier descent, cf. \cite{Katz} Theorem 5.1]\label{thm: Cartier descent}
	Let $Y$ be a smooth variety over $k$. Then the construction of $(\on{Fr}^{*}\mathcal{F},\nabla_{\on{can}})$ induces an equivalence between the category of quasi-coherent sheaves on $Y^{(1)}$ and the category of $D_Y$-modules on $Y$ with zero $p$-curvature. 
\end{theorem}

\subsection{Differential operators on smooth stacks}\label{Section: differential operators on stacks}
Let $Y$ be a smooth irreducible algebraic stack over an algebraically closed field $k$. When $k$ is the field of complex numbers $\mathbb{C}$, for $Y$ that is good in the sense that it satisfies $\on{dim}T^{*}Y=2\on{dim}Y$, the sheaf of differential operators on $Y$ is defined in \cite{BD} as a sheaf of algebras $D_Y$ on the smooth topology $Y_{sm}$. We review this definition as follows.  The objects of $Y_{sm}$ are $k$-schemes $S$ together with a smooth morphism $f_S: S\longrightarrow Y$, and the morphisms between $(S, f_S)$ and $(S', f_{S'})$ are pairs $(\phi, \alpha)$ containing a smooth morphism $\phi: S\longrightarrow S'$ and $\alpha: f_S\xrightarrow{\simeq} f_{S'} \circ \phi$. Let $(S, f_S)$ be an object of $Y_{sm}$. We denote by $\mathcal{I}$ the left ideal $D_S \mathcal{T}_{S/Y}\subset D_S$ generated by the relative tangent sheaf $\mathcal{T}_{S/Y}$. We define $(D_Y)^{\sharp}_S\coloneqq D_S/\mathcal{I}$. It has a $D_S$-action by left multiplication. Let $\mathcal{N}_{D_S}(\mathcal{I})$ be the normalizer of $\mathcal{I}$ in $D_S$. We define $(D_Y)_S\coloneqq \mathcal{N}_{D_S}(\mathcal{I})/\mathcal{I}$. In other words, we set $(D_Y)_S=\mathcal{E}nd_{D_S}((D_Y)^{\sharp}_S)^{op}$. For any smooth morphism $\phi: S\longrightarrow S'$ over $Y$, we have a canonical isomorphism
\begin{equation}\label{equation: pull-back D-module}
    \phi^{*}((D_Y)^{\sharp}_{S'})\xrightarrow{\simeq} (D_Y)^{\sharp}_{S},
\end{equation}
which restricts to an isomorphism
\begin{equation}\label{equation: pull-back differential operators}
    \phi^{-1}((D_Y)_{S'})\xrightarrow{\simeq} (D_Y)_{S},
\end{equation}
where $\phi^{-1}$ is the sheaf-theoretic inverse image. We call $D_Y$ the sheaf of differential operators on $Y$. 

It is observed in \cite{BB} that the isomorphism (\ref{equation: pull-back differential operators}) no longer holds when $k$ is of characteristic $p>0$. But meanwhile, $\on{Fr}_{*} D_Y$ is a quasi-coherent sheaf on $Y^{(1)}$, and the authors constructed a coherent sheaf of algebras $\mathcal{D}_Y$ on $T^{*}Y^{(1)}$ that satisfies $\pi^{(1)}_{*}\mathcal{D}_Y\cong \on{Fr}_{*}D_Y$. The construction of $\mathcal{D}_Y$ is as follows. For any $k$-scheme $S$ with a smooth morphism $f_S: S\longrightarrow Y$, we need to define a coherent sheaf of algebras $(\mathcal{D}_Y)_S$ on $(T^{*}Y)_S^{(1)}\coloneqq S^{(1)}\times _{Y^{(1)}} T^{*}Y^{(1)}$. We consider the $D_S$-module $(D_Y^{\sharp})_S$, and denote by $(\mathcal{D}_Y^{\sharp})_S$ the corresponding coherent sheaf on $T^{*}S^{(1)}$. Since we mod out the left ideal generated by $\mathcal{T}_{S/Y}$ when defining $(D_Y^{\sharp})_S$, the support of $(\mathcal{D}_Y^{\sharp})_S$ lies in the closed substack $(T^{*}Y)_S^{(1)}\xhookrightarrow{df_S^{(1)}} T^{*}S^{(1)}$. We set $(\mathcal{D}_Y)_S\coloneqq \mathcal{E}nd_{\mathcal{D}_S}((\mathcal{D}_Y^{\sharp})_S)^{op}$. For any smooth morphism $\phi: S\longrightarrow S'$ over $Y$, isomorphism (\ref{equation: pull-back D-module}) induces an isomorphism $(\widetilde{\phi}^{(1)})^{*}(\mathcal{D}_Y)_{S'}\cong (\mathcal{D}_Y)_{S}$, where $\widetilde{\phi}$ is the map $(T^{*}Y)_S\longrightarrow (T^{*}Y)_{S'}$. Therefore $(\mathcal{D}_Y)_S$ sheafifies to be a coherent sheaf of algebras $\mathcal{D}_Y$ on $T^{*}Y^{(1)}$. We have the following proposition:

\begin{prop}
    [cf. \cite{BB} Lemma 3.14 and \cite{Travkin} Proposition 2.7]\label{lemma: Azumaya property stack version}
    The coherent sheaf of algebras $\mathcal{D}_Y$ satisfies $\pi^{(1)}_{*}\mathcal{D}_Y\cong \on{Fr}_{*}D_Y$. If the stack $Y$ is good in the sense that $\on{dim}T^{*}{Y}=2\on{dim}Y$, and we denote by $T^{*}Y^0$ the maximal smooth open substack of $T^{*}Y$, then the restriction of $\mathcal{D}_Y$ to $(T^{*}Y^{0})^{(1)}$ is an Azumaya algebra of rank $p^{2\on{dim}Y}$. 
\end{prop}

\subsection{$\mathcal{D}$-modules, Azumaya algebras and $\mathbb{G}_m$-gerbes}\label{subsection: stack of splittings}
Let $k$ be an algebraically closed field. Let $\mathcal{B}$ be a $k$-scheme locally of finite type. Let $Y$ be a stack locally of finite type over $\mathcal{B}$. Let $\widetilde{Y}\longrightarrow Y$ be a $\mathbb{G}_m$-gerbe over $Y$. We denote by $\on{QCoh}(\widetilde{Y})$ the category of quasi-coherent sheaves on $\widetilde{Y}$. We say $\widetilde{Y}$ splits if there is an isomorphism $\widetilde{Y}\cong Y\times B\mathbb{G}_m$ of $\mathbb{G}_m$-gerbes. In this case, there is a decomposition 
\[
\on{QCoh}(\widetilde{Y})\cong \prod_{n\in \mathbb{Z}}\on{QCoh}(\widetilde{Y})_n
\]
given by the weight of the $\mathbb{G}_m$-action. If $\widetilde{Y}$ does not split, we still have such a decomposition by pulling back along the action map $a: B\mathbb{G}_m\times \widetilde{Y}\longrightarrow \widetilde{Y}$. We call $\on{QCoh}(\widetilde{Y})_1$ the category of twisted quasi-coherent sheaves associated to $\widetilde{Y}$.

Let $\mathcal{A}$ be an Azumaya algebra on $Y$. We associate with it a $\mathbb{G}_m$-gerbe $\widetilde{Y}_{\mathcal{A}}$ over $Y$, which is defined as follows. For $f: S\longrightarrow \mathcal{B}$ a map of schemes, $\widetilde{Y}_{\mathcal{A}}(S)$ classifies triples $(y, E, \sigma)$ where $y\in Y(S)$, $E$ is a vector bundle on $S$, and $\sigma: y^{*}\mathcal{A}\xrightarrow{\simeq} \mathcal{E}nd(E)$ is an isomorphism of algebras over $S$. We call $\widetilde{Y}_{\mathcal{A}}$ the stack of splittings of $\mathcal{A}$. We have the following lemma:

\begin{lemma}[cf. \cite{BB} Lemma 2.3 and \cite{Donagi-Pantev} Example 2.6]\label{lemma: gerbe}
    There is a canonical equivalence between the category $\mathcal{A}\on{-mod}$ of $\mathcal{A}$-modules on $Y$ and  $\on{QCoh}(\widetilde{Y}_{\mathcal{A}})_1$. 
\end{lemma}

Now let $Y$ be a smooth irreducible algebraic stack over $k$. A (crystalline) $\mathcal{D}$-module $\mathcal{M}$ on $Y$ is the datum of a $\mathcal{D}$-module $\mathcal{M}_S$ on $S$ for each object $(S, f_S)$ in $Y_{sm}$, and an isomorphism $\phi^{!}\mathcal{M}_{S'}\xrightarrow{\simeq} \mathcal{M}_S$ of $\mathcal{D}$-modules for each morphism $(\phi,\alpha)$, $\phi: S\longrightarrow S'$ in $Y_{sm}$. Here $\phi^{!}$ denotes the $\mathcal{O}$-module pull-back with the natural $\mathcal{D}$-module structure. Those isomorphisms need to satisfy the cocycle condition for compositions. When $k$ is of characteristic $p>0$, $\mathcal{D}$-modules on $Y$ correspond to twisted quasi-coherent sheaves associated to a certain $\mathbb{G}_m$-gerbe $\mathcal{G}_Y$ on $T^{*}Y^{(1)}$, which is defined as follows. For any smooth morphism $f_S: S\longrightarrow Y$, we associate with it a $\mathbb{G}_m$-gerbe $(\mathcal{G}_Y)_S$ on $(T^{*}Y)_S^{(1)}\coloneqq S^{(1)}\times _{Y^{(1)}} T^{*}Y^{(1)}$, which is defined to be the pull-back of the $\mathbb{G}_m$-gerbe of splittings of the Azumaya algebra $\mathcal{D}_S$ along $df_S^{(1)}: (T^{*}Y)_S^{(1)}\longrightarrow T^{*}S^{(1)}$. For any smooth morphism $\phi: S\longrightarrow S'$ over $Y$, we have an isomorphism $(\widetilde{\phi}^{(1)})^{*}(\mathcal{G}_Y)_{S'}\longrightarrow (\mathcal{G}_Y)_S$ since $df_S$ factorizes as
\[
    S\times_Y T^{*}Y=S\times_{S'}S'\times_{Y} T^{*}Y\xrightarrow{\on{Id}\times df_{S'}}S\times_{S'}T^{*}S'\xrightarrow{d\phi}T^{*}S,
\]
and the two Azumaya algebras $(d\phi^{(1)})^{*}\mathcal{D}_S$ and $p_2^{*}\mathcal{D}_{S'}$ are equivalent by Proposition \ref{Prop: pull back Azumaya}. It is shown in \cite{Travkin} that the category of $\mathcal{D}$-modules on $Y$ is equivalent to the category of twisted quasi-coherent sheaves associated to $\mathcal{G}_Y$, see Theorem 2.3 in \cite{Travkin}.

Now we assume $Y$ satisfies $\on{dim}T^{*}Y=2\on{dim}Y$, and denote by $T^{*}Y^0$ the maximal smooth open substack of $T^{*}Y$. Recall that in Section \ref{Section: differential operators on stacks}, we defined a coherent sheave of algebras $\mathcal{D}_Y$ on $T^{*}Y^{(1)}$, such that its restriction to $(T^{*}Y^0)^{(1)}$ is an Azumaya algebra of rank $p^{2\on{dim}Y}$. The $\mathbb{G}_m$-gerbe of splittings of this Azumaya algebra is isomorphic to the restriction of $\mathcal{G}_Y$ to $(T^{*}Y^0)^{(1)}$, see Proposition 2.7 in \cite{Travkin}. Therefore the category of $\mathcal{D}_Y|_{(T^{*}Y^0)^{(1)}}$-modules is a localization of the category of (crystalline) $\mathcal{D}$-modules on $Y$.

\subsection{Tensor structures on Azumaya algebras}\label{subsection: tensor structure}
Let $\mathcal{G}$ be a commutative group stack over $\mathcal{B}$, and $\mathcal{A}$ an Azumaya algebra over $\mathcal{G}$. We denote the multiplication on $\mathcal{G}$ by $\mu: \mathcal{G}\times \mathcal{G}\longrightarrow \mathcal{G}$. Following \cite{OV}, we define a tensor structure on $\mathcal{A}$ to be an equivalence of Azumaya algebras from $\mu^{*}\mathcal{A}$ to $\mathcal{A}\boxtimes \mathcal{A}$, which is a bimodule $\mathcal{M}$ that induces a Morita equivalence, together with an isomorphism
\[
    \mathcal{M}\boxtimes \mathcal{A}\otimes_{\mu^{*}\mathcal{A}\boxtimes\mathcal{A}}(\mu, p_3)^{*}\mathcal{M}\cong \mathcal{A}\boxtimes\mathcal{M}\otimes_{\mathcal{A}\boxtimes \mu^{*}\mathcal{A}}(p_1, \mu)^{*}\mathcal{M}
\]
of bimodules that satisfies the pentagon condition \cite{DM}(1.0.1). 

 A tensor structure on the Azumaya algebra $\mathcal{A}$ induces a group structure on the stack $\mathcal{Y}_{\mathcal{A}}$ of splittings of $\mathcal{A}$ as follows. Let $S$ be a $k$-scheme. An $S$-point of $\mathcal{Y}_{\mathcal{A}}$ is a pair $(a, E)$, where $a\in \mathcal{G}(S)$ and $E$ is a splitting module for $a^{*}\mathcal{A}$. Let $(a,E)$ and $(b,F)$ be two such pairs. The locally free sheaf $E\boxtimes F$ is a splitting module for $a^{*}\mathcal{A}\boxtimes b^{*}\mathcal{A}$. Applying the equivalence between $\mu^{*}\mathcal{A}$ and $\mathcal{A}\boxtimes \mathcal{A}$ and then pulling-back along the diagonal map $\Delta_S: S\longrightarrow S\times S$, we get a splitting module for $\mu(a,b)^{*}\mathcal{A}$. The construction of this group structure implies that the projection map $\mathcal{Y}_{\mathcal{A}}\longrightarrow \mathcal{G}$ is a group homomorphism, therefore we have a short exact sequence:
\[
    0\longrightarrow B\mathbb{G}_m\longrightarrow \mathcal{Y}_{\mathcal{A}}\longrightarrow \mathcal{G}
    \longrightarrow 0.
\]

\section{A non-abelian Hodge correspondence between $\mathcal{L}oc_{n,q}$ and $\mathcal{H}iggs_{n,q}$}\label{section:nah}
\subsection{Spectral data for flat connections with regular singularities}\label{section: hitchin morphism for flat connections}
Let $(E, \nabla)$ be a flat connection of rank $n$ on $X$ with regular singularity at $q$. We associate with it the $p$-curvature $\psi_{\nabla}$, which is a $\mathcal{O}_X$-linear map 
\[\psi_{\nabla} :\on{Fr}^{*}\mathcal{T}_{X^{(1)}}(-q)\longrightarrow \mathcal{E}nd(E).
\]
It is associated with the $p$-linear map 
\[
   \psi_{\nabla}: \mathcal{T}_X(-q)\longrightarrow \mathcal{E}nd(E)
\]
defined by $\psi_{\nabla}(\partial)=\nabla(\partial)^p-\nabla(\partial^{[p]})$ for any $\partial\in\mathcal{T}_X(-q)(U)$ and $U\subseteq X$ open. 
We can think of $\psi_{\nabla}$ as a twisted Higgs field
\[ \psi_{\nabla}: E\longrightarrow E\otimes \on{Fr}^{*}\Omega_{X^{(1)}}(q).
\]
The coefficients of its characteristic polynomial define a point $b$ of
\[\bigoplus_{i=1}^{n} \Gamma(X, (\on{Fr}^{*}\Omega_{X^{(1)}}(q))^{i}).
\]
Let $\on{Fr}^{*}:\bigoplus_{i=1}^{n}\Gamma(X^{(1)}, \Omega_{X^{(1)}}(q)^i)
\hookrightarrow \bigoplus_{i=1}^{n} \Gamma(X, (\on{Fr}^{*}\Omega_{X^{(1)}}(q))^{i})$ be the pull-back map. It follows from a similar argument as in \cite{LP} Proposition 3.2 that $b$ actually lies in the image of $\on{Fr}^{*}$, and we also denote by $b$ the corresponding point in  $B^{(1)}\cong \displaystyle\bigoplus_{i=1}^{n}\Gamma(X^{(1)}, \Omega_{X^{(1)}}(q)^i)$.  
We call this map 
$h': \mathcal{L}oc_{n,q}\longrightarrow B^{(1)}$ the Hitchin map for flat connections with regular singularity at $q$. The corresponding spectral curve $\Sigma_{b}'$ lies in the total space of $\on{Fr}^{*}\Omega_{X^{(1)}}(q)$, which is isomorphic to $X\times_{X^{(1)}}T^{*}X(q)^{(1)}$. Since $b\in B^{(1)}$, $\Sigma_b'$ fits into the following pull-back square:
\bd
\xymatrix{
\Sigma_{b}' \ar[r]^{\rho} \ar[d]^{\pi'} & \Sigma_{b}^{(1)} \ar[d]^{\pi^{(1)}} \\
X \ar[r]^{\on{Fr}} & X^{(1)}
}
\ed
where $\displaystyle\Sigma_b^{(1)}\subset T^{*}X(q)^{(1)}$ is the spectral curve above $b\in B^{(1)}$ as defined in Section \ref{subsection: basic constructions}. We denote by $E'\in\on{Coh}(\Sigma_b')$ the spectral sheaf corresponding to $\psi_{\nabla}$, so $E'$ satisfies $\pi'_{*}(E')\cong E$.

Let $x$ be a local parameter of $\mathcal{O}_{X,q}$. Let $(E,\nabla)$ be a flat connection with regular singularity at $q$. Restricting $\psi_{\nabla}(x\partial_x)$ to $q$, we get $\on{res}_q(\psi_{\nabla})\in \on{End}(E_q)$, which we call the residue of $\psi_{\nabla}$ at $q$. 

\begin{lemma}\label{lemma: residue of p-curvature}
	$\on{res}_q(\psi_{\nabla})=(\on{res}_q{\nabla})^p-\on{res}_q{\nabla}.$	
\end{lemma}

\begin{proof}
	This equation follows from the computation $(x\partial_x)^{[p]}=x\partial_x$.
\end{proof}

\begin{remark}\label{remark: residue of p-curvature}
	If we assume $\on{res}_q{\nabla}$ is nilpotent, since $p> n$, $(\on{res}_q{\nabla})^p=0$. So $\on{res}_q(\psi_{\nabla})=-\on{res}_q{\nabla}$. In particular, they lie in the same nilpotent orbit.
\end{remark}

\subsection{Statement of the theorem}\label{subsection: Statement of the theorem}
Let $\underline{a}$ be an unordered $n$-tuple of elements in $k$. We denote by $\mathcal{H}iggs_{n,\underline{a}}(X^{(1)})$ the moduli stack of $\Omega_{X^{(1)}}(q)$-twisted Higgs bundles $(E,\phi)$ on $X^{(1)}$ such that the unordered $n$-tuple of eigenvalues of $\on{res}_q(\phi)$ is $\underline{a}$. Let $B_{\underline{a}}^{(1)}$ be the image of $\mathcal{H}iggs_{n,\underline{a}}(X^{(1)})$ under the Hitchin map $h^{(1)}$. Note that when $\underline{a}=(0,0,\dots,0)$, $B_{\underline{a}}^{(1)}=B_{\mathcal{N}}^{(1)}$. We fix a set-theoretic section $\sigma$ of the Artin-Schreier map $k\longrightarrow k$ that maps $t$ to $t^p-t$. Let $\mathcal{L}oc_{n,\sigma(\underline{a})}$ be the substack of $\mathcal{L}oc_{n,q}$ that classifies flat connections $(E,\nabla)$ such that the unordered $n$-tuple of eigenvalues of $\on{res}_{q}({\nabla})$ is $\sigma(\underline{a})$. Note that by Lemma \ref{lemma: residue of p-curvature}, $h'(\mathcal{L}oc_{n,\sigma(\underline{a})})\subseteq B_{\underline{a}}^{(1)}$. 

We denote by $\mathcal{L}oc_{n, \sigma(\underline{a})}^r$ the substack of $\mathcal{L}oc_{n, \sigma(\underline{a})}$ that classifies flat connections $(E,\nabla)$ such that the corresponding spectral sheaf $E'\in \on{Coh}(\Sigma'_b)$ is invertible. We have the following theorem:
\begin{theorem} \label{thm: nah}
\mbox{}
	\begin{enumerate}
	\item $\mathcal{L}oc_{n, \sigma(\underline{a})}^r$ has a natural structure of a $\on{Pic}(\Sigma^{(1)}/B^{(1)}_{\underline{a}})$-torsor. 
	\item $\mathcal{L}oc_{n, \sigma(\underline{a})}\cong \mathcal{L}oc_{n, \sigma(\underline{a})}^r\times^{\on{Pic}(\Sigma^{(1)}/B^{(1)}_{\underline{a}})}\mathcal{H}iggs_{n,\underline{a}}(X^{(1)})$.
	\end{enumerate}
\end{theorem}

Before getting into the proof of Theorem \ref{thm: nah}, we state two corollaries. 
\begin{corollary}
	There exists an \'etale cover $U\longrightarrow B^{(1)}_{\underline{a}}$, such that 
	\[
	\mathcal{L}oc_{n, \sigma(\underline{a})}\times_{B^{(1)}_{\underline{a}}}U\cong \mathcal{H}iggs_{n,\underline{a}}(X^{(1)})\times_{B^{(1)}_{\underline{a}}}U.
	\]
\end{corollary}

We denote by $\mathcal{H}iggs_{\mathcal{N}}(X^{(1)})$ the moduli stack of $\Omega_{X^{(1)}}(q)$-twisted Higgs bundles $(E,\phi)$ on $X^{(1)}$ such that $\on{res}_q(\phi)$ is nilpotent, and by $\mathcal{L}oc_{\mathcal{N}}$ the substack of $\mathcal{L}oc_{n,q}$ that classifies $(E,\nabla)$ with nilpotent $\on{res}_q(\nabla)$. Then we have:
\begin{corollary}\label{cor:nah}
   \mbox{}
	\begin{enumerate}
	\item $\mathcal{L}oc_{\mathcal{N}}^r$ has a natural structure of a $\on{Pic}(\Sigma^{(1)}/B_{\mathcal{N}}^{(1)})$-torsor,
	\item $\mathcal{L}oc_{\mathcal{N}}\cong\mathcal{L}oc_{\mathcal{N}}^r\times^{\on{Pic}(\Sigma^{(1)}/B_{\mathcal{N}}^{(1)})}\mathcal{H}iggs_{\mathcal{N}}(X^{(1)})$.
	\end{enumerate}
\end{corollary}

Our definition of $\mathcal{L}oc^{r}_{n, \sigma(\underline{a})}$ and formulation of Theorem \ref{thm: nah} is motivated by the work of Chen-Zhu \cite{CZ15} on the characteristic $p$ version of the non-abelian Hodge correspondence for flat connections without singularities. The strategy of proof is similar to \cite{CZ15} besides the proof of the surjectivity result Proposition \ref{prop: surjective on k-points}. The rest of this section is devoted to the proof of Theorem \ref{thm: nah}. We start by showing: 

\begin{prop}\label{prop: surjective on k-points}
	The map $h':\mathcal{L}oc^{r}_{n, \sigma(\underline{a})}\longrightarrow B^{(1)}_{\underline{a}}$ is surjective.
\end{prop}

We need to show that for any $b\in B^{(1)}_{\underline{a}}(k)$, there exists $(E,\nabla)\in\mathcal{L}oc_{n, \sigma(\underline{a})}^r(k)$ that is mapped to $b$ under the Hitchin map. The idea of constructing $(E,\nabla)$ is as follows: we construct a flat connection $(E_{0},\nabla_{0})$ on $X\backslash q$ and a flat connection $(\hat{E},\hat{\nabla})$ on the formal disk around $q$, such that both flat connections have the correct $p$-curvature. Then we glue $(E_{0},\nabla_{0})$ and $(\hat{E},\hat{\nabla})$ together using the Beauville-Laszlo theorem \cite{BL95}. 

\subsection{Proof of Proposition \ref{prop: surjective on k-points}} Let $b\in B^{(1)}_{\underline{a}}(k)$. Let $\pi': \Sigma'\longrightarrow X$ and $\pi^{(1)}:\Sigma^{(1)}\longrightarrow X^{(1)}$ be the corresponding spectral covers as described in Section \ref{section: hitchin morphism for flat connections}. We will construct $(E,\nabla)$ such that $h'(E,\nabla)=b$ and the spectral sheaf $E'$ is invertible. 

	\emph{Step 1.} In this step, we show that there exists a flat connection $(E_{0},\nabla_{0})$ on  $X\backslash q$ such that the spectral curve of $\psi_{\nabla_{0}}$ is $\Sigma'\backslash (\pi')^{-1}(q)$ and the spectral sheaf $E_{0}'$ is invertible. Such a $(E_{0},\nabla_{0})$ is equivalent to a splitting of the Azumaya algebra $i_0^{*}\mathcal{D}_{X\backslash q}$, where $i_0$ is the embedding $i_0: \Sigma^{(1)}\backslash (\pi^{(1)})^{-1}(q)\longrightarrow T^{*}(X\backslash q)^{(1)}$. Note that for any rank $p$ vector bundle $F$ on $\Sigma^{(1)}\backslash (\pi^{(1)})^{-1}(q)$ such that $\mathcal{E}nd(F)\cong i_0^{*}\mathcal{D}_{X\backslash q}$, the corresponding spectral sheaf on $\Sigma'\backslash (\pi')^{-1}(q)$ is invertible. This is because for any $p \in X\backslash q$ and $x$ a local parameter at $p$, $x$ acts as a regular nilpotent matrix on the fiber $F_{p'}$ for any $p'\in \Sigma^{(1)}\backslash (\pi^{(1)})^{-1}(q)$ such that $\pi^{(1)}(p')=p$, see the proof of Lemma 2.2.1 in \cite{BMR}. The existence of such a splitting is guaranteed by the following theorem. 
	
\begin{theorem}[cf. \cite{Grochenig} Theorem 3.21]\label{thm: Tsen}
	Let $Y$ be a scheme of finite type over an algebraically closed field. Assume $\on{dim}(Y)\leq 1$. Then $H^{2}_{et}(Y, \mathbb{G}_m)=0$. In particular, every Azumaya algebra on $Y$ splits. 
\end{theorem}

	\emph{Step 2.} In this step, we construct a flat connection $(\hat{E},\hat{\nabla})$ on the formal disk $D=\on{Spec}(\hat{\mathcal{O}}_{X,q})$ around $q$ that satisfies the following three properties:
	\begin{enumerate}[(5.1)]
	\item $(\hat{E},\hat{\nabla})$ has regular singularity and the unordered $n$-tuple of eigenvalues of $\on{res}({\hat{\nabla}})$ is $\sigma(\underline{a})$,
	\item the spectral curve of $\psi_{\hat{\nabla}}$ is $\hat{\Sigma}'\coloneqq D \times _X \Sigma'$,
	\item the spectral sheaf $\hat{E}'$ is invertible.
	\end{enumerate}
	
Now let $x$ be a local parameter of $X$ at $q$, then $D\cong \on{Spec}(k[[x]])$. We denote $\iota(x\partial_x)=x^p\partial_x^p$ by $y$, so 
\[\mathcal{O}_{\hat{\Sigma}'}\cong k[[x]][y]/(f), \text{ where }f=y^n+b_{1}(x)y^{n-1}+\dots + b_{n-1}(x)y+b_n(x), b_i(x)\in k[[x]].
\]
Since $b\in B^{(1)}_{\underline{a}}$, $b_i(x)$ actually lies in $k[[x^p]]$. We assume that $\underline{a}$ consists of $t$ distinct elements $a_1, a_2,\dots,a_t$, each appearing $m_i$ times, then $\bar{f}\in k[[x^p]][y]/x^pk[[x^p]][y]$ factorizes as
\[
    \bar{f}=\prod_{i=1}^{t}(y-a_i)^{m_i},
\]
therefore $f$ factorizes as $f=f_1f_2\cdots f_t$, where $f_i\in k[[x^p]][y]$ is monic and $\bar{f_i}=(y-a_i)^{m_i}$. Therefore $\hat{\Sigma}'$ is the disjoint union of $\hat{\Sigma}'_i\coloneqq\on{Spec}k[[x]][y]/(f_i)$. It is enough to construct flat connections $(\hat{E}_i, \hat{\nabla}_i)$ with the following properties:
\begin{enumerate}[(5.1')]
    \item \label{surj: first}$(\hat{E}_i,\hat{\nabla}_i)$ has regular singularity and the eigenvalues of $\on{res}(\psi_{\hat{\nabla}_i})$ are all $\sigma(a_i)$.
	\item \label{surj: second}the spectral curve of $\psi_{\hat{\nabla}_i}$ is $\hat{\Sigma}'_i$,
	\item \label{surj: third}the spectral sheaf $\hat{E}'_i$ is invertible. 
\end{enumerate}

Since $\mathcal{O}_{\hat{\Sigma}'_i}$ is a local ring, (5.\ref{surj: third}') implies that $\hat{E}'_i$ is isomorphic to $\mathcal{O}_{\hat{\Sigma}'_i}$. Let $e$ be its generator. A meromorphic flat connection with spectral curve $\hat{\Sigma}'_i$ is determined by the connection acting on $e$, which can be written as $\nabla(e)=ge dx$, $g\in \mathcal{O}_{\hat{\Sigma}'_i}[x^{-1}]$. By (5.\ref{surj: first}') and (5.\ref{surj: second}'), $\nabla$ need to satisfy the following:
	\begin{enumerate}[(5.1'')]
		\item \label{local construction: first}$(\nabla(x\partial_x)-\sigma(a_i))(e)\subseteq (x, y-a_i)e$,
		\item \label{local construction: second}$(\nabla(x\partial_x)^p-\nabla((x\partial_x)^{[p]}))(e)=ye$.
	\end{enumerate}
Since $(\nabla(x\partial_x)^p-\nabla((x\partial_x)^{[p]}))(e)=x^p(\partial_x^{p-1}(g)+g^p)e$, (5.\ref{local construction: second}'') is equivalent to the following equation in $\mathcal{O}_{\hat{\Sigma}'_i}$:
\[x^p(\partial_x^{p-1}(g)+g^p)=y.
\]
We look for solutions of the form 
\[g=-(y-a_i-\sigma(a_i))/x+g_1, g_1 \in k[[x]][y], 
\]
so (5.\ref{local construction: first}'') is automatically satisfied, and (5.\ref{local construction: second}'') is equivalent to 
\begin{equation}\label{eq: local formula existence of regular connection}
    \partial_x^{p-1}(g_1)+g_1^p=(y-a_i)^p/x^p.
\end{equation}
Note that since ${f_i}\equiv (y-a_i)^{m_i}$ mod $x^pk[[x^p]][y]$, there exists a polynomial $h\in k[[x^p]][y]$ such that $(y-a_i)^p/x^p=\bar{h}$ in $\mathcal{O}_{\hat{\Sigma}'_i}$. By a substitution $y'=y-a_i$, we can assume that ${f_i}\equiv y^{m_i}$ mod $x^pk[[x^p]][y]$. In this case, $\mathcal{O}_{\hat{\Sigma}'_i}\cong k[[x,y]]/(f_i)$, therefore it is enough to find a solution $g_1\in k[[x,y]]$.
It is easy to see that for any $h\in k[[x^p]][y]$, the equation $\partial_x^{p-1}(g_1)+g_1^p=h$ has solutions in $k[[x,y]]$. We look for solutions of the form $g_1=x^{p-1} \cdot g_2$, where $g_2\in k[[x^p, y]]$. Equation \ref{eq: local formula existence of regular connection} becomes
\[
g_2^p\cdot x^{p^2-p}-g_2=h,
\]
for which the existence of solutions follows from Hensel's lemma. 

	\emph{Step 3.}
		Let $D^{\times}=\on{Spec}(k((x)))$ be the punctured disk around $q$, and let $\hat{\Sigma}^{\times}=D^{\times}\times _D {\hat{\Sigma}}$ be the spectral curve above $D^{\times}$. Both $(E_0|_{D^{\times}},\nabla_0|_{D^{\times}})$ and $(\hat{E}|_{D^{\times}}, \hat{\nabla}|_{D^{\times}})$ give splittings of the Azumaya algebra $\mathcal{D}_X|_{\hat{\Sigma}^{\times}}$. Since all invertible sheaves on $\hat{\Sigma}^{\times}$ are trivial, we have an isomorphism of connections \[(E_0|_{D^{\times}},\nabla_0|_{D^{\times}})\cong (\hat{E}|_{D^{\times}}, \hat{\nabla}|_{D^{\times}}).
		\]
		We fix such an isomorphism. By the theorem of Beauville-Laszlo \cite{BL95}, $E_0$ and $\hat{E}$ can be glued together to get a rank $n$ vector bundle $E$ on $X$. Since the gluing data is compatible with the connections, $\nabla_0$ and $\hat{\nabla}$ are glued together to get a flat connection $\nabla$ on $E$ with regular singularity at $q$. This connection $(E,\nabla)$ satisfies all the properties we need.

\subsection{$D_X$-modules on spectral covers of $X$}
Let $b\in B^{(1)}$, and let $\pi': \Sigma'\longrightarrow X$ and $\pi^{(1)}: \Sigma^{(1)}\longrightarrow X^{(1)}$ be the corresponding spectral covers as described in Section \ref{section: hitchin morphism for flat connections}. We have the following pull-back square:
\bd
\xymatrix{
\Sigma' \ar[r]^{\rho} \ar[d]^{\pi'} & \Sigma^{(1)} \ar[d]^{\pi^{(1)}} \\
X \ar[r]^{\on{Fr}} & X^{(1)}
}
\ed
There is a canonical $D_X$-action on $\mathcal{O}_{\Sigma'}=\mathcal{O}_X\otimes_{\mathcal{O}_{X^{(1)}}} \mathcal{O}_{\Sigma^{(1)}}$, which comes from the canonical action of $D_X$ on $\mathcal{O}_X$. Similarly, for any quasi-coherent sheaf $\mathcal{M}$ on $\Sigma^{(1)}$, the pull-back sheaf $\rho^{*}\mathcal{M}\cong \mathcal{O}_X\otimes_{\mathcal{O}_{X^{(1)}}} {\mathcal{M}}$ has a canonical $D_X$-action. We denote by $\nabla_{can}$ the corresponding map
\[\nabla_{can}:\rho^{*}\mathcal{M}\longrightarrow \rho^{*}\mathcal{M}\otimes_{\mathcal{O}_X}\Omega_X.
\] 
 
\begin{definition}
	We define a $D_X$-module on $\Sigma'$ to be a quasi-coherent sheaf $\mathcal{F}$ on $\Sigma'$ together with a $k$-linear map
	\[\nabla:\mathcal{F}\longrightarrow \mathcal{F}\otimes_{\mathcal{O}_X}\Omega_X
	\] 
	that satisfies $\nabla(fs)=\nabla_{can}(f)(s)+f\nabla(s)$ for $f\in \mathcal{O}_{U}$, $s\in \mathcal{F}(U)$ and $U\subseteq \Sigma'$ open. Let $D_X(-q)$ be the subsheaf of algebras of $D_X$ generated by $\mathcal{O}_X$ and $\mathcal{T}_X(-q)$. Similarly we define $D_X(-q)$-modules on $\Sigma'$. The only difference is that now $\nabla$ is a map
	\[\nabla:\mathcal{F}\longrightarrow \mathcal{F}\otimes_{\mathcal{O}_X}\Omega_X(q).
	\] 
\end{definition}
We have the following lemma concerning this definition.
\begin{lemma}\label{lemma: properties of DMOC}
\mbox{}
	\begin{enumerate}
		\item The structure sheaf $\mathcal{O}_{\Sigma'}$ is a $D_X$-module on $\Sigma'$. For any quasi-coherent sheaf $\mathcal{M}$ on $\Sigma^{(1)}$, the pull-back $\rho^{*}\mathcal{M}$ is a $D_X$-module on $\Sigma'$.
		\item Let $(E,\nabla)$ be a flat connection with regular singularity at $q$ such that $h'(E,\nabla)=b$. Let $E'\in\on{Coh}(\Sigma')$ be the corresponding spectral sheaf. Then $(E', \nabla)$ is a $D_X(-q)$-module on $\Sigma'$.
		\item Let $(\mathcal{F}_1,\nabla_1)$ and $(\mathcal{F}_2,\nabla_2)$ be two $D_X(-q)$-modules on $\Sigma'$, then $\mathcal{F}_1\otimes_{\mathcal{O}_{\Sigma'}}\mathcal{F}_2$ and $\mathcal{H}om_{\mathcal{O}_{\Sigma'}}(\mathcal{F}_1, \mathcal{F}_2)$ have canonical structures of $D_X(-q)$-modules on $\Sigma'$. 
	\end{enumerate}
\end{lemma}

In all of the cases above, we denote by $\nabla_{can}$ the corresponding map induced by the action of $\mathcal{T}_X(-q)$.

Let $(\mathcal{F},\nabla)$ be a $D_X(-q)$-module on $\Sigma'$ such that $\pi'_{*}(\mathcal{F})$ is locally free, then $\pi'_{*}(\mathcal{F})$ has the structure of a flat connection with regular singularity at $q$. 

\subsection{Proof of Theorem \ref{thm: nah}}
Now we construct the map $\Phi$ that induces the isomorphism in Theorem \ref{thm: nah}. Let $b\in B^{(1)}_{\underline{a}}$. Let $(E,\nabla_E)\in\mathcal{L}oc_{n, \sigma(\underline{a})}^r$ and $(M,\phi)\in \mathcal{H}iggs_{n,\underline{a}}(X^{(1)})$, both mapped to $b$ under the Hitchin map. We denote the spectral sheaf of $(E,\nabla_E)$ by $E'\in \on{Coh}(\Sigma')$ and the spectral sheaf of $(M,\phi)$ by $\mathcal{M}\in \on{Coh}(\Sigma^{(1)})$. 
\begin{lemma}\label{lemma, locally free is locally free}
    Let $\mathcal{G}\in\on{Coh}(\Sigma^{(1)})$ and let $\mathcal{L}$ be an invertible sheaf on $\Sigma'$. The push-forward $\pi'_{*}(\mathcal{L}\otimes\rho^{*}(\mathcal{G}))$ is a locally free sheaf of rank $n$ on $X$ if and only if $\pi^{(1)}_{*}(\mathcal{G})$ is a locally free sheaf of rank $n$ on $X^{(1)}$. 
\end{lemma}

By Lemma \ref{lemma: properties of DMOC} and Lemma \ref{lemma, locally free is locally free},  we get a flat connection $(\pi'_{*}(E'\otimes \rho^{*}(\mathcal{M})),\nabla_{can})$ on $X$ with regular singularity at $q$. 

\begin{lemma}\label{lemma: still map to b}
\mbox{}
\begin{enumerate}
    \item 
	The flat connection $(\pi'_{*}(E'\otimes \rho^{*}(\mathcal{M})),\nabla_{can})$ is mapped to $b$ under the Hitchin map $h'$,
	\item The residue $\on{res}_q(\nabla_{can})$ has eigenvalues $\sigma(\underline{a})$.
	\end{enumerate}
\end{lemma}

The construction of $(\pi'_{*}(E'\otimes \rho^{*}(\mathcal{M})),\nabla_{can})$ is functorial. Therefore we have a morphism of stacks over $B^{(1)}_{\underline{a}}$:
\[
\Phi: \mathcal{L}oc_{n, \sigma(\underline{a})}^r\times_{B^{(1)}_{\underline{a}}}{\mathcal{H}iggs_{n,\underline{a}}(X^{(1)})}\longrightarrow \mathcal{L}oc_{n, \sigma(\underline{a})}.
\]

Now we construct a map $\Psi$ in the inverse direction. Let $(F,\nabla_F)$ be a point of $\mathcal{L}oc_{n, \sigma(\underline{a})}$ such that $h'(F,\nabla_F)=b$. Let $F'\in\on{Coh}(\Sigma')$ be the spectral sheaf. Then by Lemma \ref{lemma: properties of DMOC}, there is a canonical $D_X(-q)$-action on $\pi'_{*}(\mathcal{H}om_{\mathcal{O}_{\Sigma'}}(E', \mathcal{O}_{\Sigma'})\otimes F')$. We denote $\mathcal{H}om_{\mathcal{O}_{\Sigma'}}(E', \mathcal{O}_{\Sigma'})\otimes F'$ by $\mathcal{F}$. 

\begin{lemma}\label{lemma: p-curvature 0}
\mbox{}
\begin{enumerate}
    \item The flat connection $(\pi'_{*}(\mathcal{F}), \nabla_{can})$ has zero $p$-curvature. 
    \item The residue $\on{res}_q(\nabla_{can})$ is nilpotent. 
\end{enumerate}
\end{lemma}

By Lemma \ref{lemma: residue of p-curvature}, the residue $\on{res}_q(\nabla_{can})$ of $(\pi'_{*}(\mathcal{F}), \nabla_{can})$ at $q$ satisfies \[\on{res}_q(\nabla_{can})^p-\on{res}_q(\nabla_{can})=0.
\]
This implies $\on{res}_q(\nabla_{can})$ is a semisimple matrix with integer eigenvalues. But meanwhile, $\on{res}_q(\nabla_{can})$ needs to be nilpotent, so $\on{res}_q(\nabla_{can})$ must be the zero, therefore $(\pi'_{*}(\mathcal{F}), \nabla_{can})$ is a flat connection without singularities. By the Cartier descent (Theorem \ref{thm: Cartier descent}), there is a canonical quasi-coherent sheaf $\mathcal{N}$ on $X^{(1)}$ such that $(\pi'_{*}(\mathcal{F}), \nabla_{can})$ is isomorphic to $(\on{Fr}^{*}(\mathcal{N}), \nabla_{can})$. Note that $\mathcal{N}$ can be identified with elements in $\pi'_{*}(\mathcal{F})$ that vanish under $\nabla_{can}$. The action of $\mathcal{O}_{\Sigma^{(1)}}$ preserve those elements, therefore there is a canonical quasi-coherent sheaf $\mathcal{M}$ on $\Sigma^{(1)}$ such that $(\mathcal{F}, \nabla_{can})$ is isomorphic to $(\rho^{*}(\mathcal{M}), \nabla_{can})$ as $D_X(-q)$-modules on $\Sigma'$.  Since $E'$ is an invertible sheaf, $(F', \nabla_F)\cong (E'\otimes \rho^{*}(\mathcal{M}), \nabla_{can})$. The construction of $\mathcal{M}$ is functorial. Therefore, we have a morphism $\Psi$ of stacks over $B^{(1)}_{\underline{a}}$:
\[\Psi: \mathcal{L}oc_{n, \sigma(\underline{a})}^r\times_{B^{(1)}_{\underline{a}}}\mathcal{L}oc_{n, \sigma(\underline{a})} \longrightarrow \mathcal{H}iggs_{n,\underline{a}}(X^{(1)}).
\]

Let $(E,\nabla_E)\in\mathcal{L}oc_{n,\sigma(\underline{a})}^r$ such that $h'(E,\nabla_E)=b$, and denote the corresponding spectral sheaf by $E'\in\on{Coh}(\Sigma')$. Let $\mathcal{L}$ be an invertible sheaf on $\Sigma^{(1)}$. Then by Lemma \ref{lemma: still map to b}, 
\[(\pi'_{*}(E'\otimes \rho^{*}(\mathcal{L})),\nabla_{can}) \in \mathcal{L}oc_{n, \sigma(\underline{a})}^r.
\]
This construction defines an action of $\on{Pic}(\Sigma^{(1)}/B^{(1)}_{\underline{a}})$ on $\mathcal{L}oc_{n, \sigma(\underline{a})}^r$. 

\begin{prop}\label{prop: pseudo-torsor}
    This action induces the structure of a pseudo $\on{Pic}(\Sigma^{(1)}/B^{(1)}_{\underline{a}})$-torsor on $\mathcal{L}oc_{n, \sigma(\underline{a})}^r$.
\end{prop}

\begin{proof}
Let $S$ be a $k$-scheme. Let $b$ be an $S$-point of $B^{(1)}_{\underline{a}}$. We need to show that the action of $\on{Pic}(\Sigma_b^{(1)})$ on the fiber $(\mathcal{L}oc_{n,\sigma(\underline{a})}^r)_{b}\coloneqq \mathcal{L}oc_{n,\sigma(\underline{a})}^r\times_{B^{(1)}_{\underline{a}},b}{S}$ is simply transitive when $(\mathcal{L}oc_{n,\sigma(\underline{a})}^r)_{b}$ is non-empty. Let $(E,\nabla_E)$ and $(F,\nabla_F)$ be two points of $\mathcal{L}oc_{n, \sigma(\underline{a})}^r$ that is mapped to $b$ under the Hitchin map. We denote the corresponding spectral sheaves by $E'$ and $F'$. By the discussion after Lemma \ref{lemma: p-curvature 0}, there exists a quasi-coherent sheaf $\mathcal{M}$ on $\Sigma_b^{(1)}$ such that $(F', \nabla_F)\cong (E'\otimes \rho^{*}(\mathcal{M}), \nabla_{can})$. Since $\rho$ is faithfully flat, $E'$ and $F'$ being invertible sheaves on $\Sigma_b'$ implies that $\mathcal{M}$ is an invertible sheaf on $\Sigma_b^{(1)}$. The map $\Phi$ induces a map 
\[
    \on{Hom}_{\mathcal{O}_{\Sigma_b^{(1)}}}(\mathcal{O}_{\Sigma_b^{(1)}},\mathcal{M})\xrightarrow{\simeq}\on{Hom}_{D_X(-q)}((E,\nabla_E), (F,\nabla_F)),
\]
which is an isomorphism since $\Psi$ produces its inverse.
\end{proof}

We denote by $\mathcal{L}oc_{n, q}$ is the moduli stack of flat connections on $X$ with regular singularity at $q$, without constraints on the eigenvalues of the residue. Let $\mathcal{L}oc^r_{n, q}\subset \mathcal{L}oc_{n, q}$ be the substack characterized by the spectral sheaf being invertible. We have the following proposition:  

\begin{prop}\label{prop: torsor is smooth}
    The map $h':\mathcal{L}oc^{r}_{n, q}\longrightarrow B^{(1)}$ is smooth. 
\end{prop}

Before getting into the proof of Proposition \ref{prop: torsor is smooth}, we state a corollary that is going to be used in the proof of Theorem \ref{thm: nah}.

\begin{corollary}\label{cor: torsor is smooth}
    The map $h':\mathcal{L}oc^{r}_{n, \sigma(\underline{a})}\longrightarrow B_{\underline{a}}^{(1)}$ is smooth. 
\end{corollary}
\begin{proof}    
    The map 
    $\mathcal{L}oc^{r}_{n, q}\times_{B^{(1)}}B_{\underline{a}}^{(1)}\longrightarrow B_{\underline{a}}^{(1)}
    $
    is smooth by base change, and the fiber product
    $
    \mathcal{L}oc^{r}_{n, q}\times_{B^{(1)}}B_{\underline{a}}^{(1)}$ is the disjoint union of $\mathcal{L}oc^{r}_{n, \underline{c}}$, where $\underline{c}$ ranges from all unordered $n$-tuples of elements in $k$ that maps to $\underline{a}$ under the Artin-Schreier map. 
\end{proof}

We denote by $\widetilde{\mathcal{L}oc}_{n,q}$ the stack that classifies triples $(E,\nabla, \theta)$, where $E$ is a vector bundle of rank $n$ on $X$, $\nabla: E\longrightarrow E\otimes \Omega_{X}(q)$ is a flat connection with regular singularity at $q$ and $\theta: E_q\xrightarrow{\cong} k^n$ is a frame of $E$ at $q$. 
The natural action of $GL_n$ on the frame $\theta$ gives $\widetilde{\mathcal{L}oc}_{n,q}$ the structure of a $GL_n$-torsor over $\mathcal{L}oc_{n,q}$.

\begin{lemma}
     $\mathcal{L}oc_{n,q}$ and $\widetilde{\mathcal{L}oc}_{n,q}$ are algebraic stacks locally of finite type over $k$.
\end{lemma}
\begin{proof}
    The 1-morphism $\mathcal{L}oc_{n,q}\longrightarrow \on{Bun}_n$ is representable and locally of finite presentation. Since $\on{Bun}_n$ is an algebraic stack locally of finite type over $k$ and $\widetilde{\mathcal{L}oc}_{n,q}$ is a $GL_n$-torsor over $\mathcal{L}oc_{n,q}$, both $\mathcal{L}oc_{n,q}$ and $\widetilde{\mathcal{L}oc}_{n,q}$ are algebraic stacks locally of finite type over $k$.
\end{proof}
\begin{lemma}\label{Lemma: torsor smooth}
   $\mathcal{L}oc^r_{n,q}$ and $\widetilde{\mathcal{L}oc}_{n,q}^r$ are smooth. 
\end{lemma}
\begin{proof}
In order to show that $\widetilde{\mathcal{L}oc}_{n,q}^r$ is smooth, all we need to show is that for any small extension of finite-generated Artinian local $k$-algebras $A'\longrightarrow A$, an $A$-point of $\widetilde{\mathcal{L}oc}_{n,q}^r$ can be lifted to an $A'$-point of $\widetilde{\mathcal{L}oc}_{n,q}^r$, i.e. we want to produce the dashed arrow for the following commutative diagram:

\bd
\xymatrix{
\on{Spec}(A) \ar[r] \ar[d] &  \widetilde{\mathcal{L}oc}_{n,q}^r \\
\on{Spec}(A') \ar@{.>}[ur]
}
\ed

We denote by $(E,\nabla, \theta)$ the $k$-point
\[
    \on{Spec}(A/m_{A}A)\longrightarrow \on{Spec}(A)\longrightarrow \widetilde{\mathcal{L}oc}_{n,q}^r.
\]
The obstruction to the existence of such liftings lies in the second hypercohomology $\mathbb{H}^2(\mathscr{F}^{\bullet}_{E,\nabla})$ of the complex
\[
    \mathscr{F}^{\bullet}_{E,\nabla}: \mathcal{E}nd(E)(-q)\xrightarrow{\nabla_{\mathcal{E}nd(E)}} \mathcal{E}nd(E)\otimes \Omega_X(q),
\]
where $\nabla_{\mathcal{E}nd(E)}$ is the canonical connection on $\mathcal{E}nd(E)$ induced by $\nabla$. By Serre duality, $\mathbb{H}^2(\mathscr{F}^{\bullet}_{E,\nabla})\cong \mathbb{H}^0(\mathscr{F}^{\bullet}_{E,\nabla})$. Note that $\mathbb{H}^0(\mathscr{F}^{\bullet}_{E,\nabla})$ is isomorphic to the Lie algebra of $\on{Aut}(E,\nabla,\theta)$. Since $(E,\nabla)\in\mathcal{L}oc_{n,q}^r$, we have $\on{Aut}(E,\nabla)=k^{\times}$ by Proposition \ref{prop: pseudo-torsor}. But multiplication by scalars does not preserve the framing $\theta$, therefore $\on{Aut}(E,\nabla,\theta)$ is the trivial group. This implies $\mathbb{H}^2(\mathscr{F}^{\bullet}_{E,\nabla})\cong \mathbb{H}^0(\mathscr{F}^{\bullet}_{E,\nabla})=0$. 
\end{proof}

A by-product of the proof of Lemma \ref{Lemma: torsor smooth} is the following computation of the dimension of $\mathcal{L}oc_{n,q}^r$. Since $\mathbb{H}^0(\mathscr{F}^{\bullet}_{E,\nabla})=\mathbb{H}^2(\mathscr{F}^{\bullet}_{E,\nabla})=0$, 
\begin{align*}
   \on{dim}\widetilde{\mathcal{L}oc}_{n,q}^r=&\on{dim}\mathbb{H}^1(\mathscr{F}^{\bullet}_{E,\nabla})\\
   =&2(\on{dim}H^{0}(X, \mathcal{E}nd(E)\otimes \Omega_X(q))-\on{dim}H^{0}(X, \mathcal{E}nd(E)(-q))). 
\end{align*}
By Riemann-Roch, 
\[\on{dim}H^{0}(X, \mathcal{E}nd(E)\otimes \Omega_X(q))-\on{dim}H^{0}(X, \mathcal{E}nd(E)(-q))=n^{2}g. 
\]
Therefore
\begin{align*}
    \on{dim}\mathcal{L}oc_{n,q}^r=&\on{dim}\widetilde{\mathcal{L}oc}_{n,q}^r-\on{dim}GL_n(k)\\
    =&n^2(2g-1).
\end{align*}

\begin{proof}[Proof of Proposition \ref{prop: torsor is smooth}]
    Let $b\in B^{(1)}(k)$. By Proposition \ref{prop: surjective on k-points} and \ref{prop: pseudo-torsor}, the fiber $(\mathcal{L}oc_{n,q}^r)_{b}\coloneqq \mathcal{L}oc_{n,q}^r\times_{B^{(1)},b}{\on{Spec}(k)}$ is a $\on{Pic}(\Sigma_{b}^{(1)})$-torsor. We compute that 
    \[\on{dim}B^{(1)}=n(n+1)(2g-1)/2+n(1-g)
    \]
    and 
    \[
        \on{dim}\on{Pic}(\Sigma_{b}^{(1)})=g_{\Sigma_{b}^{(1)}}-1=n(n-1)(2g-1)/2+n(g-1),
    \]
    therefore
    \[
        \on{dim}\on{Pic}(\Sigma_{b}^{(1)})= \on{dim}\mathcal{L}oc_{n,q}^r-\on{dim}B^{(1)}.
    \]
    Since both $\mathcal{L}oc_{n,q}^r$ and $B^{(1)}$ are smooth, the map $h'$ is flat by miracle flatness. Furthermore, Since $\on{Pic}(\Sigma_{b}^{(1)})$ is smooth, $(\mathcal{L}oc_{n,q}^r)_{b}$ is smooth, therefore $h'$ is smooth.  
\end{proof}

\begin{proof}[Proof of Theorem \ref{thm: nah}]
The first part follows from Proposition \ref{prop: pseudo-torsor} and Corollary \ref{cor: torsor is smooth}. 
For the second part, it is easy to see that the morphism $\Phi$ defined above induces a morphism
\[
		\Phi: \mathcal{L}oc_{n, \sigma(\underline{a})}^r\times^{\on{Pic}(\Sigma^{(1)}/B_{\underline{a}}^{(1)})}\mathcal{H}iggs_{n,\underline{a}}(X^{(1)})\longrightarrow \mathcal{L}oc_{n, \sigma(\underline{a})},
\]
and $\Psi$ induces the inverse.
\end{proof}

Now we discuss how the residues of Higgs bundles and flat connections match under $\Phi$. 

\begin{prop}\label{prop: match of nilpotent orbits}
	Let $(E, \nabla_E)\in\mathcal{L}oc_{n,\sigma(\underline{a})}^r$ and $(M,\phi_M)\in\mathcal{H}iggs_{n,\underline{a}}(X^{(1)})$ such that 
	\[h'(E, \nabla_E)=h^{(1)}(M, \phi_M)=b\in B_{\underline{a}}^{(1)}(k).
	\]
	Denote the image of $(E, \nabla_E)$ and $(M, \phi_M)$ under $\Phi$ by $(F,\nabla_F)$. Then $\on{res}_q(\psi_{\nabla_F})$ and $\on{res}_q(\phi_M)$ lie in the same adjoint orbit.
	
\end{prop}
\begin{proof}
	Let $E'\in\on{Coh}(\Sigma')$ be the spectral sheaf of $(E,\nabla_E)$, and let $\mathcal{M}\in\on{Coh}(\Sigma^{(1)})$ be the spectral sheaf of $(M,\phi_M)$. Let $x$ be a local parameter of $X$ at $q$. Note that $\on{res}_q(\phi_M)$ is the action of $x^p\partial_x^p$ on the fiber $\pi^{(1)}_{*}(\mathcal{M})|_q$, and $\on{res}_q(\psi_{\nabla_F})$ is the action of $x^p\partial_x^p$ on the fiber $\pi'_{*}(E'\otimes \rho^{*}(\mathcal{M}))|_q$. Since $\pi^{(1)}_{*}(\mathcal{M})|_q\cong \pi'_{*}(\rho^{*}(\mathcal{M}))|_q$ with the same $x^p\partial_x^p$ action, it suffices to show that the action of $x^p\partial_x^p$ on $\pi'_{*}(\rho^{*}(\mathcal{M}))|_q$ and $\pi'_{*}(E'\otimes\rho^{*}(\mathcal{M}))|_q$ lie in the same adjoint orbit. This follows from the assumption that $E'$ is an invertible sheaf. 
\end{proof}

In particular, if $\sigma(\underline{a})=(\underline{0})$, Proposition \ref{prop: match of nilpotent orbits} together with Remark \ref{remark: residue of p-curvature} implies that $\on{res}_q(\nabla_F)$ and $\on{res}_q(\phi_M)$ lie in the same nilpotent orbit. Therefore we have the following
\begin{corollary}
    The scheme-theoretic image of $\mathcal{L}oc_{n,P}$ under the Hitchin map $h'$ is $B_P^{(1)}$.
\end{corollary}

\section{Tamely ramified geometric Langlands correspondence in positive characteristic}\label{section: main theorem}
\subsection{The algebra $\mathcal{D}_{\on{Bun}_{n,P}}$}\label{subsection: statement of the main theorem}\label{subsection: definition of the algebra}
In this subsection we clarify what we mean by $\mathcal{D}_{\on{Bun}_{n,P}}$. Since the stack $\on{Bun}_{n,P}$ does not satisfy the property required in Proposition \ref{lemma: Azumaya property stack version}, we cannot apply this proposition directly. In order to solve this problem, we introduce a new stack $\underline{\on{Bun}}_{n, P}$ similar to the stack $\underline{\on{Bun}}_n$ introduced in \cite{BB}. The stack $\underline{\on{Bun}}_{n, P}$ classifies the same objects as $\on{Bun}_{n, P}$, but the morphisms are different. Let $S$ be a $k$-scheme, and let $(E, E^{\bullet}_q)$ and $(F, F^{\bullet}_q)$ be two rank $n$ vector bundles on $S\times X$ with partial flag structures of type $P$ (see Remark \ref{remark: parabolic subgroup}) along $S\times q$, then the set of morphisms between $(E, E^{\bullet}_q)$ and $(F, F^{\bullet}_q)$ are defined to be the set of isomorphic classes of pairs $(\iota, \mathcal{L})$, where $\mathcal{L}$ is a line bundle on $S$ and $\iota$ is an isomorphism $\iota: (E, E^{\bullet}_q)\xrightarrow{\simeq} (F\otimes p_S^{*}(\mathcal{L}), F^{\bullet}_q\otimes \mathcal{L})$. By taking $\mathcal{L}=\mathcal{O}_S$, we get a natural map $\on{Bun}_{n, P}\longrightarrow \underline{\on{Bun}}_{n, P}$, and $\on{Bun}_{n, P}$ is a $\mathbb{G}_m$-gerbe over $\underline{\on{Bun}}_{n, P}$. 
\begin{prop}\label{prop: the stack is good}
    The stack $\underline{\on{Bun}}_{n,P}$ satisfies $\on{dim}T^{*}\underline{\on{Bun}}_{n,P}=2\on{dim}\underline{\on{Bun}}_{n,P}$.
\end{prop}
\begin{proof}
  We apply the same strategy as in \cite{Ginzburg}. The main goal is to show that the nilpotent cone $\mathcal{N}ilp\coloneqq h_P^{-1}(0)\subset T^{*}\underline{\on{Bun}}_{n,P}$ is isotropic. Then the argument used in the proof of Proposition 7, 8 in \cite{Ginzburg} applies here to deduce the desired equality. Let $B$ be a Borel subgroup of $GL_n(k)$ that is contained in $P$. We denote by $\on{Bun}_B$ the moduli stack of $B$-bundles on $X$. By Lemma 23 in \cite{Heinloth}, the natural map $f: \on{Bun}_B\longrightarrow \on{Bun}_{n,P}$ is surjective. In order to apply Lemma 5 in \cite{Ginzburg} to show that $\mathcal{N}ilp$ is isotropic, all we need to show is that for any $(E, E^{\bullet}_q, \phi)\in \mathcal{N}ilp(k)$, there exists $E_B\in \on{Bun}_B(k)$ such that $f(E_B)=(E,E^{\bullet}_q)$ and $f^{*}(\phi)=0\in T^{*}_{E_B}\on{Bun}_B$, i.e. there exists a complete flag structure of $E$ over $X$ such that its restriction to $q$ is compatible with the partial flag structure $E^{\bullet}_q$, and the Higgs field $\phi$ is nilpotent with respect to this complete flag structure. We choose a basis $(e_1, e_2,\dots, e_n)$ of $E_q$ such that the complete flag structure 
  \[
       0\subset \langle e_1\rangle\subset\langle e_1, e_2\rangle\subset \cdots \langle e_1, e_2, \dots, e_n \rangle=E_q
  \]
  is compatible with $E_q^{\bullet}$. Let $U=\on{Spec}A$ be an open neighborhood of $q$ over which $E$ and $\Omega^1_X(q)$ trivializes. Fixing such trivializations, the Higgs field $\phi$ corresponds to an $A$-linear map $A^n\longrightarrow A^n$. Since $\on{res}_q(\phi)$ is nilpotent with respect to $E_q^{\bullet}$, $\phi_q(e_i)$ lies in the $k$-vector space spanned by $e_1, e_2, \dots, e_{i-1}$. Shrinking $U$ if necessary, the basis $(e_1, e_2,\dots, e_n)$ of $E_q$ can be lifted to a basis $(\tilde{e}_1, \tilde{e}_2,\dots, \tilde{e}_n)$ of $E$ over $U$ that still satisfies $\phi(\tilde{e}_i)\in \langle \tilde{e}_1, \tilde{e}_2, \dots, \tilde{e}_{i-1}\rangle.$
The $B$-reduction of $E$ over $U$ given by
\[
     0\subset \langle \tilde{e}_1\rangle\subset\langle \tilde{e}_1, \tilde{e}_2\rangle\subset \cdots \langle \tilde{e}_1, \tilde{e}_2, \dots, \tilde{e}_n \rangle=E|_U
\]
can be extended to a $B$-reduction over $X$ since $GL_n(k)/B$ is projective. Such a $B$-reduction satisfies all the properties we need.
\end{proof}

\begin{remark}
Over $\mathbb{C}$ the field of complex numbers, the analogue of Proposition \ref{prop: the stack is good} was proved in \cite{BKV17} (see Theorem 6, 7) for a general reductive group $G$ and parahoric $P$. It is not clear to the author if their arguments can be adapted to the characteristic $p$ setting.
\end{remark}

Now we apply Proposition \ref{lemma: Azumaya property stack version} to $\underline{\on{Bun}}_{n,P}$ and get $\mathcal{D}_{\underline{\on{Bun}}_{n,P}}$. The sheaf of algebras $\mathcal{D}_{\on{Bun}_{n,P}}$ is defined to be the pull-back of $\mathcal{D}_{\underline{\on{Bun}}_{n,P}}$ to $\on{Bun}_{n,P}$. 

We denote $\mathcal{H}iggs_{n,P}\times_{B_{P}}B^{0}_P$ by $\mathcal{H}iggs^{0}_{n,P}$ and  $\mathcal{L}oc_{n,P}\times_{B_P^{(1)}}(B^{0}_P)^{(1)}$ by $\mathcal{L}oc_{n,P}^{0}$. Since $\mathcal{H}iggs^{0}_{n,P}$ is smooth, $\mathcal{D}_{{\on{Bun}}_{n,P}}$ restricts to an Azumaya algebra $\mathcal{D}^{0}_{{\on{Bun}}_{n,P}}$ on
\[(\mathcal{H}iggs^{0}_{n,P})^{(1)}\subseteq \mathcal{H}iggs_{n,P}^{(1)}\cong  T^{*}{\on{Bun}}_{n,P}^{(1)}.
\]
Now we are in the position to state the main theorem of the paper:
\begin{theorem}\label{thm: GLP}
	There exists an $\mathcal{O}_{\mathcal{L}oc^0_{n,P}}\boxtimes \mathcal{D}^0_{\on{Bun}_{n,P}}$-module $\mathcal{P}$, such that the Fourier-Mukai transform $\Phi_{\mathcal{P}}$ with kernel $\mathcal{P}$ induces an equivalence
	\[ D^{b}(\on{QCoh}({\mathcal{L}oc_{n,P}^{0}}))\xrightarrow{\simeq}D^{b}(\mathcal{D}^{0}_{{\on{Bun}}_{n,P}}\on{-mod})
	\]
	between the bounded derived category of quasi-coherent sheaves on $\mathcal{L}oc_{n,P}^{0}$ and the bounded derived category of $\mathcal{D}^{0}_{{\on{Bun}}_{n,P}}$-modules.
\end{theorem}

\subsection{The tensor structure on $\mathcal{D}^{0}_{\on{Bun}_{n,P}}$}\label{subsection: group structure on Azumaya algebra}
Recall that in Section \ref{section: Spectral data of parabolic Higgs bundles}, we constructed a family of curves $\widetilde{\Sigma}\longrightarrow B_P^{0}$ such that 
\[
        \mathcal{H}iggs^{0}_{n,P}\cong \on{Pic}(\widetilde{\Sigma}/B_P^{0}).
\]
In this subsection we show that there is a natural tensor structure on the Azumaya algebra $\mathcal{D}^{0}_{\on{Bun}_{n,P}}$, in the sense of \cite{OV} (see Section \ref{subsection: tensor structure}). We denote $\widetilde{\Sigma}\backslash \widetilde{\pi}^{-1}(B_P^0\times q)$ by $\widetilde{\Sigma}^{0}$, where $\widetilde{\pi}:  \widetilde{\Sigma}\longrightarrow B_P^{0}\times X$ is the universal spectral cover.
Let $i$ be the natural inclusion
\[
    i: \widetilde{\Sigma}^{0}\longrightarrow B_P^{0}\times T^{*}X.
\]
We denote by $a$ the morphism
\[
    a:  \widetilde{\Sigma}\times_{B_P^0} \on{Pic}(\widetilde{\Sigma}/B_{P}^0) \longrightarrow \on{Pic}(\widetilde{\Sigma}/B_{P}^0)
\]
that maps $(\widetilde{x},L)$ to $L(\widetilde{x})$. 
We denote by $\kappa$ the Abel-Jacobi map
\[
    \kappa: \widetilde{\Sigma}/B_P^0\longrightarrow \on{Pic}(\widetilde{\Sigma}/B_{P}^0)
\]  
that maps $\widetilde{x}\in \widetilde{\Sigma}/B_P^0$ to $\mathcal{O}_{\widetilde{\Sigma}}(\widetilde{x})$. 

Let $\theta_X$ be the tautological 1-form on $T^{*}X$ and $\theta_{{\on{Bun}}_{n,P}}$ the tautological 1-form on $T^{*}{\on{Bun}}_{n,P}$. By similar arguments as in Theorem 4.12 in \cite{BB}, we have
\begin{prop} \label{prop: 1-form}
When restricted to $\widetilde{\Sigma}^0\times_{B_P^0} \on{Pic}(\widetilde{\Sigma}/B_{P}^0)$,
\[
   i^{*}\theta_X\boxtimes \theta_{{\on{Bun}}_{n,P}}= a^{*}\theta_{{\on{Bun}}_{n,P}}|_{\widetilde{\Sigma}^0\times_{B_P^0} \on{Pic}(\widetilde{\Sigma}/B_{P}^0)}.
\]
In particular,  
\[
    i^{*}\theta_X=\kappa^{*}\theta_{{\on{Bun}}_{n,P}}|_{\widetilde{\Sigma}^{0}}.
\]
\end{prop}

For the proof of Proposition \ref{prop: 1-form}, we consider the  moduli stack $\mathcal{H}ecke^1_P$ of quadruples \[((E,E^{\bullet}_q),(F, F^{\bullet}_q),x,i: E\hookrightarrow F),\]
where $x\in X\backslash q$, $(E,E^{\bullet}_q), (F, F^{\bullet}_q)\in \on{Bun}_{n,P}$ such that $F/E$ is the simple skyscraper sheaf at $x$, and the partial flag structures $E^{\bullet}_q$ and $F^{\bullet}_q$ coincide under $i$. By considering $\on{Im}(i_x)\subset F_x$, this data is equivalent to a triple $((F, F^{\bullet}_q), x, V\subset F_x)$, where $V$ is a dimension $n-1$ subspace of $F_x$. We consider the following projections:
\bd
\xymatrix{
& \mathcal{H}ecke^1_P\ar[dl]_{q} \ar[dr]^{p}\\
\on{Bun}_{n,P} & & X\backslash q\times \on{Bun}_{n,P}
}
\ed
where $q$ maps the quadruple to $(F, F^{\bullet}_q)$ and $p$ maps the quadruple to $((E, E_q^{\bullet}),x)$. Both $p$ and $q$ are smooth.

Consider the following pull-back diagram:
\bd
\xymatrix{
Z^0\ar[rr]^{f_1} \ar[d]^{f_2} &&  q^{*}\mathcal{H}iggs^0_{n,P}\ar[d]^{dq}\\
p^{*}(T^{*}(X\backslash q)\times \mathcal{H}iggs^0_{n,P})\ar[rr]^{dp} && T^{*}\mathcal{H}ecke^1_P
}
\ed
We define $\alpha_1$ to be the map:
\[
    \alpha_1=\on{pr}_2\circ f_1: Z^0\longrightarrow \mathcal{H}iggs^0_{n,P},
\]
where $\on{pr}_2$ is the projection $q^{*}\mathcal{H}iggs^0_{n,P}=\mathcal{H}ecke^1_P\times_{q,\on{Bun}_{n,P}}\mathcal{H}iggs^0_{n,P}\longrightarrow \mathcal{H}iggs^0_{n,P}$. Similarly we define 
\[\alpha_2=\on{pr}_2\circ f_2: Z^0\longrightarrow T^{*}(X\backslash q)\times \mathcal{H}iggs^0_{n,P}.
\]
The stack $Z^0$ and the maps $\alpha_1$, $\alpha_2$ can be described as follows:

\begin{lemma}
    The stack $Z^0$ is isomorphic to $\widetilde{\Sigma}^0\times_{B_P^0}\mathcal{H}iggs_{n,P}^0$. Under this isomorphism, $\alpha_1$ corresponds to the addition map $a$, and $\alpha_2$ corresponds to the product of the projection map $\widetilde{\Sigma}^0\subset T^{*}(X\backslash q) \times B_P^0\xrightarrow{\on{pr}_1} T^{*}(X\backslash q)$ with the identity map of $\mathcal{H}iggs_{n,P}^0$.
\end{lemma}
\begin{proof}
    Let $((E,E^{\bullet}_q),(F, F^{\bullet}_q),x,i: E\hookrightarrow F)$ be a $k$-point of $\mathcal{H}ecke^1_P$ which we denote by $\tau$. There is a short exact sequence of cotangent spaces
\[
    0\longrightarrow T_x^{*}X\xrightarrow{p_X^{*}} T_{\tau}^{*}\mathcal{H}ecke^1_P\xrightarrow{\pi} T_\tau^{*}p_X^{-1}(x)\longrightarrow 0,
\]
where $p_X$ is the projection $\mathcal{H}ecke^1_P\longrightarrow X$. The fiber $p_X^{-1}(x)$ classifies $(F, F^{\bullet}_q)\in\on{Bun}_{n,P}$ together with a subspace $V\subset F_x$ of dimension $n-1$. Therefore $T_\tau^{*}p_X^{-1}(x)$ is the subspace of twisted Higgs fields $\phi\in \Gamma(X, \mathcal{E}nd(F)\otimes\Omega_X(q+x))$ such that $\on{res}_q(\phi)$ is nilpotent with respect to the partial flag structure $F_q^{\bullet}$ and $\on{res}_x(\phi)$ is nilpotent with respect to $V\subset F_x$. The composite $\pi\circ dq$ maps $(\tau, (F,F_q^{\bullet}, \phi_F))$ to $\phi_F$, and $\pi\circ dp$ maps $(\tau, (E, E_q^{\bullet}, \phi_E), (x,\xi))$ to the unique extension of $\phi_E$ to $F$. Therefore $Z^0$ classifies triples 
\[((F, F_q^{\bullet}, \phi_F), x, E\subset F), 
\]
where $(F, F_q^{\bullet}, \phi_F)\in \mathcal{H}iggs_{n,P}$, $x\in X\backslash q$ such that $F/E=k_x$ and $\phi_F$ restricts to a twisted Higgs field on $E$ with no pole at $x$. Since $(F, F_q^{\bullet}, \phi_F)$ is isomorphic to $(E, E_q^{\bullet}, \phi_E)$ away from $x$, they are mapped to the same point $b\in B_P^0$ under the Hitchin map. Let $\mathcal{L}$ resp. $\mathcal{L}'\in \on{Pic}(\widetilde{\Sigma}_b)$ be the invertible sheaf corresponding to $(E, E_q^{\bullet}, \phi_E)$ resp. $(F, F_q^{\bullet}, \phi_F)$ under the isomorphism in Theorem \ref{thm: spectral data family}. Since $F/E=k_x$, $\mathcal{L}'/\mathcal{L}=k_{x'}$ for some $x'\in \widetilde{\Sigma}_b^0$ that maps to $x$ under the spectral cover map. Therefore having a triple $((F, F_q^{\bullet}, \phi_F), x, E\subset F)$ as above is equivalent to having $(b, \mathcal{L}, x')$, where $b\in B_P^0$, $\mathcal{L}\in \on{Pic}(\widetilde{\Sigma}_b)$ and $x'\in \widetilde{\Sigma}_b^0$. 
\end{proof}
\begin{proof}[Proof of Proposition \ref{prop: 1-form}]
The goal is to show $\alpha_1^{*}\theta_{\on{Bun}_{n,P}}=\alpha_2^{*}(\theta_{\on{Bun}_{n,P}} \boxtimes \theta_X)$. Both 1-forms are equal to the pull-back of the tautological 1-form on $T^*\mathcal{H}ecke^1_P$ to $Z^0$.
\end{proof}

Let $\theta^{0}_{{\on{Bun}}_{n,P}}$ be the restriction of $\theta_{{\on{Bun}}_{n,P}}$ to $\mathcal{H}iggs^0_{n,P}$. By Lemma 3.14 in \cite{BB}, in order to construct a tensor structure on $\mathcal{D}^{0}_{\on{Bun}_{n,P}}$, it is enough to show that for the addition map 
\[ m:\on{Pic}(\widetilde{\Sigma}/B_P^{0})\times \on{Pic}(\widetilde{\Sigma}/B_P^{0})\longrightarrow \on{Pic}(\widetilde{\Sigma}/B_P^{0}),
\]
the 1-form $\theta^{0}_{{\on{Bun}}_{n,P}}$ satisfies the following equality:
\begin{equation} \label{eq:1-form}
m^{*}\theta^{0}_{{\on{Bun}}_{n,P}}=\theta^{0}_{{\on{Bun}}_{n,P}}\boxtimes \theta^{0}_{{\on{Bun}}_{n,P}}.
\end{equation}

We denote by $\on{Pic}^{d}(\widetilde{\Sigma}/B^{0}_P)$ the degree $d$ component of $\on{Pic}(\widetilde{\Sigma}/B^{0}_P)$. Since there are isomorphisms between components of $\on{Pic}(\widetilde{\Sigma}/B_P^{0})$ that preserve $\theta^{0}_{{\on{Bun}}_{n,P}}$, it is enough to prove the equality (\ref{eq:1-form}) for large enough $d$, $d'$ and 
\[
    m_{d,d'}:\on{Pic}^{d}(\widetilde{\Sigma}/B_P^{0})\times \on{Pic}^{d'}(\widetilde{\Sigma}/B_P^{0})\longrightarrow \on{Pic}^{d+d'}(\widetilde{\Sigma}/B_P^{0}).
\]
We denote by $\kappa_d$ the map
\[
   \kappa_d: (\widetilde{\Sigma}/B_P^{0})^d\longrightarrow \on{Pic}^{d}(\widetilde{\Sigma}/B_P^{0})
\]
that maps $(\widetilde{x}_1, \widetilde{x}_2, \dots, \widetilde{x}_d)\in (\widetilde{\Sigma}/B_P^{0})^d$ to $\mathcal{O}_{\widetilde{\Sigma}}(\widetilde{x}_1+\widetilde{x}_2+\cdots+\widetilde{x}_d)$. For $d>2g_{\Sigma_b}-2$, there is an open subset of $\widetilde{\Sigma}^d$ such that $\kappa_d$ is smooth and dominant. Therefore it is enough to show that
\[
    \kappa_{d+d'}^{*}\theta^{0}_{{\on{Bun}}_{n,P}}=\kappa_d^{*}\theta^{0}_{{\on{Bun}}_{n,P}}\boxtimes\kappa_{d'}^{*}\theta^{0}_{{\on{Bun}}_{n,P}}. 
\] 
By Proposition \ref{prop: 1-form}, this equality holds on $(\widetilde{\Sigma}^{0})^{d+d'}$, therefore it holds on $\widetilde{\Sigma}^{d+d'}$. 
\subsection{Torsor structure on $\mathcal{L}oc_{n,P}^{0}$}
By Corollary \ref{cor:nah}, Proposition \ref{prop: match of nilpotent orbits}, Remark \ref{remark: residue of p-curvature}, Remark \ref{remark: Richardson orbit} and Theorem \ref{thm: spectral data family}, we have the following:

\begin{prop}\label{prop: parabolic nah}
\mbox{}
\begin{enumerate}
    
    \item 

    The isomorphism in Corollary \ref{cor:nah} induces \[\mathcal{L}oc_{n,P}^{0}\cong\mathcal{L}oc_{\mathcal{N}}^r\times_{B^{(1)}}(B_P^0)^{(1)}\times^{\on{Pic}(\Sigma^{(1)}/(B_P^{0})^{(1)})}(\mathcal{H}iggs^{0}_{n,P})^{(1)},
    \]
    \item
    the action of $\on{Pic}(\widetilde{\Sigma}^{(1)}/(B_P^{0})^{(1)})$ on $(\mathcal{H}iggs^{0}_{n,P})^{(1)}$ gives $\mathcal{L}oc_{n,P}^{0}$ the structure of a $\on{Pic}(\widetilde{\Sigma}^{(1)}/(B_P^{0})^{(1)})$-torsor. 
    
\end{enumerate}
\end{prop}

Let $S$ be a $k$-scheme. Let $b$ be an $S$-point of $(B_P^{0})^{(1)}$. Consider the following commutative diagram:
\[
\xymatrix{
\widetilde{\Sigma}_b \ar@/_/[dddr]_{\widetilde{\pi}} \ar@/^/[drrr]^{\on{Fr}_{\widetilde{\Sigma}_b}} \ar@{.>}[dr]|-{\exists \widetilde{\tau}}\\
&\widetilde{\Sigma}_b' \ar[dd]^{\widetilde{\pi}'} \ar[rr]^{\widetilde{\rho}} \ar[dr]^{\sigma'} && \widetilde{\Sigma}^{(1)}_{b}\ar[dd]^(.35){\widetilde{\pi}^{(1)}}  \ar[dr]^{\sigma^{(1)}}\\
&&\Sigma_b'\ar[rr]^(.35){\rho}\ar[ld]^{\pi'} &&\Sigma^{(1)}_{b}\ar[ld]^{\pi^{(1)}}\\
&X \ar[rr]^{\on{Fr_{X}}} &&X^{(1)}}
\]
Here $\widetilde{\Sigma}'_b\coloneqq X\times_{X^{(1)}}\widetilde{\Sigma}^{(1)}_{b}$. There exists a unique map from $\widetilde{\Sigma}_b$ to $\widetilde{\Sigma}'_b$ that makes the diagram commute. We call this map $\widetilde{\tau}$. Note that $\widetilde{\tau}$ is finite since $\on{Fr}_{\widetilde{\Sigma}_b}$ is finite and $\widetilde{\rho}$ is separated. 

Let $(E,E^{\bullet}_q, \nabla_E)$ be an $S$-point of $\mathcal{L}oc_{n,P}$ such that $h'(E,\nabla_E)=b$. Let $E'\in \on{Coh}(\Sigma'_b)$ be the spectral sheaf. We associate with it an invertible sheaf $\widetilde{E}'\in  \on{Coh}(\widetilde{\Sigma}'_b)$ that satisfies $\sigma'_{*}(\widetilde{E}')=E'$ as follows. Let $(E_1,\nabla_1)\in \mathcal{L}oc^r_{\mathcal{N}}$ and $(E_2,\phi_2)\in \mathcal({H}iggs^{0}_{n,P})^{(1)}$ such that $h'(E_1,\nabla_1)=h^{(1)}(E_2,\phi_2)=b$ and they are mapped to $(E,\nabla_E)$ under the isomorphism in Proposition \ref{prop: parabolic nah}(1). Let $E'_1\in \on{Coh}(\Sigma'_b)$ be the spectral sheaf of $(E_1,\nabla_1)$ and $\mathcal{L}\in \on{Coh}(\Sigma^{(1)}_{b})$ the spectral sheaf of $(E_2,\phi_2)$, then we have $E'\cong E_1\otimes \rho^{*}\mathcal{L}$. By Theorem \ref{thm: spectral data family}, there exists a unique invertible sheaf $\widetilde{\mathcal{L}}$ on $\widetilde{\Sigma}^{(1)}_{b}$ such that $\sigma^{(1)}_{*}\widetilde{\mathcal{L}}=\mathcal{L}$. Now we define 
\[
    \widetilde{E}'=\sigma'^{*}(E_1')\otimes \widetilde{\rho}^{*}\widetilde{\mathcal{L}}. 
\]
This construction does not depend on the choice of $(E_1,\nabla_1)$ and $(E_2,\phi_2)$. 

\begin{lemma}
The flat connection $\nabla_{can}$ on $\widetilde{\tau}^{*}\widetilde{E}'=\mathcal{O}_{\widetilde{\Sigma}_b}\otimes_{\mathcal{O}_{\widetilde{\Sigma}_b'}} \widetilde{E}'$ defined by
	\[ \nabla_{can}(\partial)(f\otimes s)=\partial(f)\otimes s + f \nabla_E({d\widetilde{\pi}^{*}}(\partial))(s)
	\]
	for any $\partial\in \mathcal{T}_{U}$, $f\in \mathcal{O}_{U}$, $s\in \widetilde{E}'(\widetilde{\tau}(U))$ and open $U\subseteq\widetilde{\Sigma}_b$, has no singularities. Here $d\widetilde{\pi}^{*}$ is the tangent map $\mathcal{T}_{\widetilde{\Sigma}_b}\longrightarrow \widetilde{\pi}^{*}\mathcal{T}_X$.
\end{lemma}
\begin{proof}
	Since $b\in B_P^0$, $\widetilde{\pi}^{-1}(q)$ consists of $r$ points $q_1,q_2,\dots,q_r$. 
	The only places that $\nabla_{can}$ might have singularities are $q_1,q_2,\dots,q_r$. Note that $\widetilde{\pi}: \widetilde{\Sigma}_b\longrightarrow X$ has ramification index $\lambda_i$ at $q_i$. Let $t$ be a local parameter at $q_i\in \widetilde{\Sigma}_b$ and $x$ a local parameter at $q\in X$ such that $\widetilde{\pi}^{*}(x)=t^{\lambda_i}$. Let $U$ be an open neighborhood of $q_i$. Since $(E, \nabla_E)$ has nilpotent residue at $q$, there exists a positive integer $N$ such that for any $m\geq N$ and $s\in \widetilde{E}'(\widetilde{\tau}(U))$, $(\nabla_E(x\partial_x))^m(s)\in x\widetilde{E}'$. We compute that 
	\[
	(\nabla_{can}(t\partial_t))^{pN}(f\otimes s)=(t\partial_t)^{pN}(f)\otimes s+f (\nabla_E({d\widetilde{\pi}^{*}}(t\partial_t)))^{pN}(s).
	\]
	Since $(t\partial_t)^{pN}(f)\in t\mathcal{O}_{\widetilde{\Sigma}_b}$ and $(\nabla_E(\on{d\widetilde{\pi}^{*}}(t\partial_t)))^{pN}(s)=(\nabla_E(\lambda_ix\partial_x))^{pN}(s)\in x\widetilde{E}'$, the sum lies in $t\mathcal{O}_{\widetilde{\Sigma}_b}\otimes\widetilde{E}'$. Therefore the residue of $(\widetilde{\tau}^{*}\widetilde{E}', \nabla_{can})$ at $q'$ is nilpotent. But since  $\widetilde{\tau}^{*}\widetilde{E}'$ is an invertible sheaf on $\widetilde{\Sigma}_b$, the residue must be zero, so the flat connection $\nabla_{can}$ has no singularity at $q_i$.   
\end{proof}

We denote by $\widetilde{\theta}_b$ the restriction of $\kappa^{*}\theta_{\on{\on{Bun}}_{n,P}}$ to $\widetilde{\Sigma}_b$. Then we have the following:

\begin{lemma}
    The connection $(\widetilde{\tau}^{*}\widetilde{E}',\nabla_{can})$ is mapped to $\widetilde{\theta}_b^{(1)}$ under the Hitchin map $h'$. 
\end{lemma}
\begin{proof}
    Let $p' \in \widetilde{\Sigma}_b$ such that $\widetilde{\pi}(p')=p\neq q$. Let $x$ be a local parameter at $p$. We denote $\partial_x$ by $y$, so near $p$ the spectral curve $\widetilde{\Sigma}_b$ is the vanishing scheme of a polynomial of the form 
   \[ y^n+b_1(x)y^{n-1}+\cdots+b_{n-1}(x)y+b_n(x).
   \]
   Since $\widetilde{\Sigma}_b$ is smooth, $y-y(p')$ is a local parameter of $\widetilde{\Sigma}_b$ at $p'$. Since $\theta_X=ydx$, we have 
   \[
    i^{*}\theta_X=\partial_y(x) ydy \text{ and }  (i^{*}\theta_X)^{(1)}=(\partial_y(x))^py^p dy^p.
   \]
   Let $U$ be an open neighborhood of $p'$. For $f\in \mathcal{O}_{U}$ and $s\in \widetilde{E}'(\widetilde{\tau}(U))$, we compute
   \begin{align*}
        \partial_y^p(f\otimes s) &=f\otimes (d\pi(\partial_y))^p(s)\\
         &=f\otimes (\partial_y(x))^p \partial_x^p(s)\\
         &=f\otimes (\partial_y(x))^p y^p(s)\\
         &=<\partial_{y^p}, (i^{*}\theta_X)^{(1)}>(f\otimes s).
    \end{align*}
    Therefore $h'(\widetilde{\tau}^{*}\widetilde{E}',\nabla_{can})$ is equal to $(i^{*}\theta_X)^{(1)}$ when restricted to $(\widetilde{\Sigma}_b^{0})^{(1)}$. By Proposition \ref{prop: 1-form}, $h'(\widetilde{\tau}^{*}\widetilde{E}',\nabla_{can})=\widetilde{\theta}_b^{(1)}$. 
\end{proof}
Now recall that for a smooth variety $Y$ over $k$, $\mathcal{D}_{Y}$ is the Azumaya algebra on $T^{*}Y^{(1)}$ that satisfies $\on{Fr}_{*}(D_Y)=\pi_{*}^{(1)}(\mathcal{D}_{Y})$. 

\begin{prop}\label{prop: key}
    Let $b$ be an $S$-point of $(B_{P}^0)^{(1)}$. The construction of $(\widetilde{\tau}^{*}\widetilde{E}',\nabla_{can})$ induces an isomorphism of stacks between $(\mathcal{L}oc_{n,P}^{0})_{b}\coloneqq \mathcal{L}oc_{n,P}^{0}\times_{(B_P^0)^{(1)},b}S$ and the stack of splittings of the Azumaya algebra $(\widetilde{\theta}_b^{(1)})^{*}\mathcal{D}_{\widetilde{\Sigma}_{b}}$. Here we think of $\widetilde{\theta}_b^{(1)}$ as a map
    \[
        \widetilde{\theta}_b^{(1)}: \widetilde{\Sigma}_{b}^{(1)}\longrightarrow T^{*}\widetilde{\Sigma}_{b}^{(1)}. 
    \]
\end{prop}
\begin{proof}
    Both stacks are $\on{Pic}(\widetilde{\Sigma}_{b}^{(1)})$-torsors. Since $\widetilde{\tau}^{*}$ is compatible with the $\on{Pic}(\widetilde{\Sigma}_{b}^{(1)})$-actions, it induces an isomorphism between those two stacks. 
\end{proof}

\subsection{Fourier-Mukai transforms on commutative group stacks}\label{subsection: FM}
In this subsection we review the Fourier-Mukai transforms on commutative group stacks, following \cite{BB}. Let $k$ be an algebraically closed field. Let $\mathcal{B}$ be an irreducible $k$-scheme that is locally of finite type. 
Let $\mathcal{G}$ be a commutative group stack locally of finite type over $\mathcal{B}$. The dual commutative group stack $\mathcal{G}^{\vee}$ classifies 1-morphisms of group stacks from $\mathcal{G}$ to $B\mathbb{G}_m$. The main examples we are going to consider are:

\begin{examples}
\mbox{}
    \begin{enumerate}
        \item $\mathcal{G}=\mathbb{Z}$, $\mathcal{G}^{\vee}=B\mathbb{G}_m$,
        \item $\mathcal{G}=B\mathbb{G}_m$, $\mathcal{G}^{\vee}=\mathbb{Z}$, 
        \item $\mathcal{G}=\mathbb{Z}_n$,  $\mathcal{G}^{\vee}=B\mu_n$. Here $\mu_n=\on{Spec}(\mathbb{Z}[x]/(x^n-1))$,
        \item $\mathcal{G}=\mu_n$,  $\mathcal{G}^{\vee}=B\mathbb{Z}_n$,
        \item $\mathcal{G}=A$ is an abelian scheme, then $\mathcal{G}^{\vee}=A^{\vee}$ is the dual abelian scheme. 
    \end{enumerate}
\end{examples}
By the definition of $\mathcal{G}^{\vee}$, there is a universal $\mathbb{G}_m$-torsor on $\mathcal{G}\times \mathcal{G}^{\vee}$, which gives rise to the Poincar\'e line bundle $\mathcal{P}_{\mathcal{G}}$.

In \cite{BB}, a commutative group stack $\mathcal{G}$ is called very nice, if locally in smooth topology, $\mathcal{G}$ is a finite product of stacks in the examples above. Under this assumption, the natural map $\mathcal{G}\longrightarrow \mathcal{G}^{{\vee}{\vee}}$ is an isomorphism. Therefore there is another Poincar\'e line bundle $\mathcal{P}_{\mathcal{G}^{\vee}}$ on $\mathcal{G}^{\vee}\times \mathcal{G}$.

\begin{theorem}[cf. \cite{BB} Theorem 2.7]
    Let $\mathcal{G}$ be a very nice commutative group stack and let $\mathcal{G}^{\vee}$ be its dual. Then the Fourier-Mukai functor $\Phi_{\mathcal{P}_{\mathcal{G}}}$ with kernel $\mathcal{P}_{\mathcal{G}}$ induces an equivalence of derived categories
    \[
        D^{b}(\on{QCoh}(\mathcal{G}))\xrightarrow{\simeq} D^{b}(\on{QCoh}(\mathcal{G}^{\vee})).
    \]
\end{theorem}

Now Let $\widetilde{\mathcal{G}}$ and $\mathcal{G}$ be very nice commutative group stacks that fit into a short exact sequence of group stacks:
\[
    0\longrightarrow B\mathbb{G}_m\longrightarrow \widetilde{\mathcal{G}}\longrightarrow \mathcal{G}\longrightarrow 0.
\]
By taking dual, we get another short exact sequence:

\[
    0\longrightarrow \mathcal{G}^{\vee}\longrightarrow \widetilde{\mathcal{G}}^{\vee}\xrightarrow{\pi} \mathbb{Z}\longrightarrow 0. 
\]
Let $\widetilde{\mathcal{G}}^{\vee}_1=\pi^{-1}(1)$. 
\begin{remark}\label{remark: gerb splitting}
Note that $\widetilde{\mathcal{G}}^{\vee}_1$ classifies maps of group stacks $\widetilde{\mathcal{G}}\longrightarrow B\mathbb{G}_m$ such that the composition
\[
    B\mathbb{G}_m\longrightarrow \widetilde{\mathcal{G}}\longrightarrow B\mathbb{G}_m
\]
is the identity. Such a map gives a splitting of $\widetilde{\mathcal{G}}$ considered as a $\mathbb{G}_m$-gerbe over $\mathcal{G}$.
\end{remark}

Recall that the $B\mathbb{G}_m$-action on $\widetilde{\mathcal{G}}$ gives a decomposition 
\[
D^{b}(\on{QCoh}(\widetilde{\mathcal{G}}))\cong \prod_{n\in \mathbb{Z}}D^{b}(\on{QCoh}(\widetilde{\mathcal{G}}))_n.
\] 
\begin{prop}[cf. \cite{Arinkin} Proposition A.7 and \cite{BB} Proposition 2.9]\label{prop: FM}
\mbox{}
    The Fourier-Mukai functor $\Phi_{\mathcal{P}_{\widetilde{\mathcal{G}}^{\vee}}}$ restricts to an equivalence of derived categories
    \[
             D^{b}(\on{QCoh}(\widetilde{\mathcal{G}}^{\vee}_1)\xrightarrow{\simeq} D^{b}(\on{QCoh}(\widetilde{\mathcal{G}}))_1. 
    \]
\end{prop}
	
\subsection{Proof of Theorem \ref{thm: GLP}}\label{subsection: proof of main theorem}
Let $\mathcal{Y}_{\mathcal{D}^0_{\on{Bun}_{n,P}}}$ be the $\mathbb{G}_m$-gerbe (defined in Subsection \ref{subsection: stack of splittings}) over $\mathcal{H}iggs^{0}_{n,P}\cong \on {Pic}(\widetilde{\Sigma}^{(1)}/(B_P^{0})^{(1)})$ that classifies splittings of the Azumaya algebra $\mathcal{D}^0_{\on{Bun}_{n,P}}$. As discussed in Section \ref{subsection: tensor structure}, the tensor structure on $\mathcal{D}^0_{\on{Bun}_{n,P}}$ gives $\mathcal{Y}_{\mathcal{D}^0_{\on{Bun}_{n,P}}}$ the structure of a commutative group stack, and it fits into a short exact sequence
\[
    0\longrightarrow B\mathbb{G}_m\longrightarrow \mathcal{Y}_{\mathcal{D}^0_{\on{Bun}_{n,P}}} \longrightarrow \on {Pic}(\widetilde{\Sigma}^{(1)}/(B_P^{0})^{(1)})\longrightarrow 0. 
\]
By taking dual, we get another short exact sequence:
\[
    0\longrightarrow \on {Pic}(\widetilde{\Sigma}^{(1)}/(B_P^{0})^{(1)})\longrightarrow \mathcal{Y}_{\mathcal{D}^0_{\on{Bun}_{n,P}}}^{\vee}\xrightarrow{\pi} \mathbb{Z}\longrightarrow 0. 
\]

\begin{prop}\label{prop: main thm torsor dual}
    $(\mathcal{Y}_{\mathcal{D}^0_{\on{Bun}_{n,P}}}^{\vee})_1\coloneqq \pi^{-1}(1)$ is isomorphic to $\mathcal{L}oc_{n,P}^{0}$ as ${Pic}(\widetilde{\Sigma}^{(1)}/(B_P^{0})^{(1)})$-torsors. 
\end{prop}
\begin{proof}
    It is enough to construct a morphism from $(\mathcal{Y}_{\mathcal{D}^0_{\on{Bun}_{n,P}}}^{\vee})_1$ to $\mathcal{L}oc_{n,P}^{0}$ that is compatible with the $\on{Pic}(\widetilde{\Sigma}^{(1)}/(B_P^{0})^{(1)})$-actions. Let $S$ be a $k$-scheme. Let $b$ be an $S$-point of $(B_P^{0})^{(1)}$. By Remark \ref{remark: gerb splitting}, an $S$-point of  $(\mathcal{Y}_{\mathcal{D}^0_{\on{Bun}_{n,P}}}^{\vee})_1$ lying above $b$ gives a splitting of the Azumaya algebra $\mathcal{D}^0_{\on{Bun}_{n,P}}|_{\on{Pic}(\widetilde{\Sigma}_b^{(1)})}$. Pulling back along the Abel-Jacobi map
    \[
        \kappa^{(1)}: \widetilde{\Sigma}_b^{(1)}\longrightarrow \on{Pic}(\widetilde{\Sigma}_b^{(1)}),
    \]
    by Corollary \ref{cor: 1-form}(2) and Proposition \ref{prop: 1-form}, such a splitting gives a splitting of the Azumaya algebra $(\widetilde{\theta}_b^{(1)})^{*}\mathcal{D}_{\widetilde{\Sigma}_b}$, which in turn gives a point of $(\mathcal{L}oc_{n,P})_b$ by Proposition \ref{prop: key}. This map is clearly compatible with the $\on{Pic}(\widetilde{\Sigma}_b^{(1)})$-actions. 
\end{proof}

Now let $\mathcal{P}_{\mathcal{Y}^{\vee}}$ be the Poincar\'e line bundle on $\mathcal{Y}_{\mathcal{D}^0_{\on{Bun}_{n,P}}}^{\vee}\times \mathcal{Y}_{\mathcal{D}^0_{\on{Bun}_{n,P}}}$. By Lemma \ref{lemma: gerbe} and Proposition \ref{prop: main thm torsor dual}, $\mathcal{P}_{\mathcal{Y}^{\vee}}$ restricts to an $\mathcal{O}_{\mathcal{L}oc_{n,P}^{0}}\boxtimes\mathcal{D}^0_{\on{Bun}_{n,P}}$-module $\mathcal{P}$. By Proposition \ref{prop: FM}, the Fourier-Mukai transform with kernel $\mathcal{P}$
\[  \Phi_{\mathcal{P}}: D^{b}(\on{QCoh}({\mathcal{L}oc_{n,P}^{0}}))\longrightarrow D^{b}(\mathcal{D}^{0}_{{\on{Bun}}_{n,P}}\on{-mod})
\]
induces an equivalence of derived categories. This completes the proof of Theorem \ref{thm: GLP}.

\subsection{The Hecke functor}\label{subsection: Hecke}
Recall that we define $\mathcal{H}ecke^1_P$ to be the moduli stack of quadruples \[((E,E^{\bullet}_q),(F, F^{\bullet}_q),x,i: E\hookrightarrow F),\]
where $x\in X\backslash q$, $(E,E^{\bullet}_q), (F, F^{\bullet}_q)\in \on{Bun}_{n,P}$ such that $F/E$ is the simple skyscraper sheaf at $x$, and the partial flag structures $E^{\bullet}_q$ and $F^{\bullet}_q$ coincide under $i$. We consider the following projections:
\bd
\xymatrix{
& \mathcal{H}ecke^1_P\ar[dl]_{q} \ar[dr]^{p}\\
\on{Bun}_{n,P} & &\on{Bun}_{n,P}\times X\backslash q
}
\ed
where $q$ maps the quadruple to $(F, F^{\bullet}_q)$ and $p$ maps the quadruple to $((E, E_q^{\bullet}),x)$. 
The Hecke functor $\on{H}_P^0$ is defined by
\[
    \on{H}_P^0: D^{b}(\mathcal{D}^0_{\on{Bun}_{n,P}}\on{-mod})\longrightarrow D^{b}(\mathcal{D}^0_{\on{Bun}_{n,P}}\boxtimes \mathcal{D}_{X\backslash q}\on{-mod})
\]
\[
    \mathcal{M}\mapsto p_{*}q^{!}\mathcal{M}.
\]
Let $\mathcal{E}$ be the universal $\mathcal{O}_{\mathcal{L}oc^0_{n,P}}\boxtimes \mathcal{D}_{X\backslash q}$-module. We define another functor $\on{W}^0_P$:
\[
    \on{W}^0_P: D^{b}(\mathcal{O}_{\mathcal{L}oc^0_{n,P}}\on{-mod})\longrightarrow D^{b}(\mathcal{O}_{\mathcal{L}oc^0_{n,P}}\boxtimes \mathcal{D}_{X\backslash q}\on{-mod})
\]
\[
    \mathcal{F}\mapsto p_1^{*}\mathcal{F}\otimes \mathcal{E},
\] 
where $p_1$ is the projection $\mathcal{L}oc_n\times X\longrightarrow \mathcal{L}oc_n$.
Let $\Phi_{\mathcal{P},X\backslash q}$ be the Fourier-Mukai equivalence induced by the pull-back of $\mathcal{P}$:
\[
\Phi_{\mathcal{P},X\backslash q}: D^{b}(\mathcal{O}_{\mathcal{L}oc^0_{n,P}}\boxtimes \mathcal{D}_{X\backslash q}\on{-mod})\xrightarrow{\simeq} D^{b}(\mathcal{D}^0_{\on{Bun}_{n,P}}\boxtimes \mathcal{D}_{X\backslash q}\on{-mod}).
\]
then we have:
\begin{theorem}\label{thm: hecke}
There is an isomorphism of functors:
    \[\on{H}_{P}^0\circ \Phi_{\mathcal{P}}\cong\Phi_{\mathcal{P},X\backslash q}\circ\on{W}^0_{P} .\]
\end{theorem}

\begin{proof}
The proof is similar to the proof of Theorem 5.4 in \cite{BB}. Since the equivalence $\Phi_{\mathcal{P}}: D^{b}(\mathcal{O}_{\mathcal{L}oc^0_{n,P}}\on{-mod})\xrightarrow{\simeq} D^{b}(\mathcal{D}^0_{\on{Bun}_{n,P}}\on{-mod})$ is the Fourier-Mukai functor with kernel the  $\mathcal{D}^0_{\on{Bun}_{n,P}}\boxtimes \mathcal{O}_{\mathcal{L}oc^0_{n,P}}$-module $\mathcal{P}$, it is enough to show that $H_P^0(\mathcal{P})$ and $W_P^0(\mathcal{P})$ are isomorphic as $\mathcal{D}^0_{\on{Bun}_{n,P}}\boxtimes \mathcal{D}_{X\backslash  q}\boxtimes \mathcal{O}_{\mathcal{L}oc^0_{n,P}}$-modules. Recall that in the proof of Proposition \ref{prop: 1-form}, we considered the pull-back diagram 
\bd
\xymatrix{
Z^0\ar[rr]^{f_1} \ar[d]^{f_2} &&  q^{*}\mathcal{H}iggs^0_{n,P}\ar[d]^{dq}\\
p^{*}(\mathcal{H}iggs^0_{n,P}\times T^{*}(X\backslash q)) \ar[rr]^{dp} && T^{*}\mathcal{H}ecke^1_P
}
\ed
and two maps $\alpha_1=\on{pr}_2\circ f_1: Z^0\longrightarrow \mathcal{H}iggs^0_{n,P}$ and 
$\alpha_2=\on{pr}_2\circ f_2: Z^0\longrightarrow \mathcal{H}iggs^0_{n,P}\times T^{*}(X\backslash q)$. Since $\alpha_1^{*}\theta_{\on{Bun}_{n,P}}=\alpha_2^{*}(\theta_{\on{Bun}_{n,P}} \boxtimes \theta_X)$, we have a canonical equivalence of Azumaya algebras by Corollary \ref{cor: 1-form} (2):
\begin{equation}\label{equation: Azumaya equivalence}
    (\alpha_1^{(1)})^{*}\mathcal{D}^0_{\on{Bun}_{n,P}}\sim (\alpha_2^{(1)})^{*}(\mathcal{D}^0_{\on{Bun}_{n,P}}\boxtimes \mathcal{D}_{X\backslash  q}). 
\end{equation}
For any $\mathcal{M}\in\mathcal{D}^0_{\on{Bun}_{n,P}}\on{-mod}$, $H^0_{P}(\mathcal{M})$ can be obtained by pulling-back along $\alpha_1^{(1)}$, applying equivalence (\ref{equation: Azumaya equivalence}), then pushing-forward along $\alpha_2^{(1)}$. For any $\sigma \in \mathcal{L}oc^0_{n,P}$, the $\mathcal{D}^0_{\on{Bun}_{n,P}}$-module $\mathcal{P}_\sigma$ is a splitting of the Azumaya algebra $\mathcal{D}^0_{\on{Bun}_{n,P}}|_{\on{Pic}(\widetilde{\Sigma}_b^{(1)})}$ that is compatible with the tensor structure defined in Section \ref{subsection: group structure on Azumaya algebra}. There is a canonical equivalence of Azumaya algebras
\begin{equation}\label{equation: equivalence 2.0}
    (\kappa_b^{(1)})^{*}\mathcal{D}^0_{\on{Bun}_{n,P}}|_{\on{Pic}(\widetilde{\Sigma}_b^{(1)})}\sim(i_b^{(1)})^{*}\mathcal{D}_{X\backslash  q}
\end{equation}
induced by the equality in Proposition \ref{prop: 1-form}, where $b=h'(\sigma)$, $\kappa_b$ is the Abel-Jacobi map $\widetilde{\Sigma}_b\longrightarrow \on{Pic}(\widetilde{\Sigma}_b)$ and $i_b$ is the inclusion $\widetilde{\Sigma}_b^0\subset T^{*}(X\backslash q)$. The $\mathcal{D}_{X\backslash  q}$-module $\mathcal{E}_\sigma$ can be obtained from $\mathcal{P}_\sigma$ by pulling-back along $\kappa_b$ and applying equivalence (\ref{equation: equivalence 2.0}). Since the stack $Z^0$ is isomorphic to $\widetilde{\Sigma}^0\times_{B_P^0}\mathcal{H}iggs_{n,P}^0$ and $\alpha_1$ corresponds to the addition map $a$, we have $H^0_P(\mathcal{P}_{\sigma})\cong \mathcal{P}_{\sigma}\boxtimes \mathcal{E}_{\sigma}$, which is what we wish to show. 
\end{proof}

Now let $(E,\nabla)$ be a $k$-point of $\mathcal{L}oc^0_{n,P}$. We denote by $\mathcal{M}_{E,\nabla}$ the image of $(E,\nabla)$ under $\Phi_{\mathcal{P}}$. By Theorem \ref{thm: hecke}, $\mathcal{M}_{E,\nabla}$ satisfies
\begin{equation*}\label{eq: hecke}
       \on{H}^0_{P}(\mathcal{M}_{E,\nabla})\cong\mathcal{M}_{E,\nabla}\boxtimes E. 
\end{equation*}

\end{document}